\documentclass[11pt]{article}

\usepackage[margin=1in]{geometry}
\usepackage{authblk}
\usepackage{graphicx}
\usepackage{multirow}
\usepackage{amsmath,amssymb,amsfonts}
\usepackage{amsthm}
\usepackage{xcolor}
\usepackage{booktabs}
\usepackage{algorithm}
\usepackage{algpseudocode}
\usepackage{subcaption}
\usepackage{caption}
\usepackage{listings}
\usepackage{cases}
\usepackage{bm}
\usepackage{placeins}
\usepackage[colorlinks=true,linkcolor=blue,citecolor=blue,urlcolor=blue]{hyperref}

\theoremstyle{plain}
\newtheorem{theorem}{Theorem}[section]

\newtheorem{lemma}[theorem]{Lemma}
\theoremstyle{definition}
\newtheorem{example}{Example}[section]
\newtheorem{definition}{Definition}[section]
\theoremstyle{remark}
\newtheorem{remark}{Remark}[section]

\newcommand{\bem}{\begin{bmatrix}}\newcommand{\eem}{\end{bmatrix}}

\def\mcala{\mathcal A}
\def\mcalb{\mathcal B}
\def\mcali{\mathcal I}

\def\bcirc{{\rm bcirc}}
\def\unfold{{\rm unfold}}
\def\fold{{\rm fold}}

\def\ran{{\rm R}}
\def\nul{{\rm null}}

\def\mbbr{\mathbb R}

\def\fm#1{\mathcal{#1}}
\def\m#1{\boldsymbol{#1}}

\def\r{\mathbb{R}}

\def \a#1{\fm A_{:, #1, :}}
\def \b#1{\fm B_{#1, :, :}}
\def \ai#1{\fm A_{#1, :, :}}
\def \z#1#2{\fm Z^{(#1)}_{#2, :, :}}
\def \bp#1{\fm B^{\perp}_{#1, :, :}}
\def \x#1{\fm X^{(#1)}}

\title{A Tensor Greedy Double-Block Extended Kaczmarz Method for Inconsistent Tensor Linear Systems under the t-product}
\author[1]{Jérémie Mabiala}
\author[1,2]{Lionel Tondji}
\affil[1]{African Institute of Mathematical Sciences and African Master's of Machine Intelligence, Mbour, Senegal}
\affil[2]{Institute for Analysis and Algebra, TU Braunschweig, Braunschweig, Germany}
\date{}

\begin{document}

\maketitle

\begin{abstract}
The randomized extended Kaczmarz method is an effective iterative framework for solving large-scale inconsistent linear systems. In this paper, we extend this framework to third-order inconsistent tensor linear systems under the t-product and propose the Tensor Greedy Double Block Extended Kaczmarz (TGDBEK) method. At each iteration, TGDBEK dynamically constructs active blocks of row and column slices via a residual-based greedy selection strategy, prioritizing the slices associated with the largest residual norms. Unlike existing tensor block Kaczmarz variants that rely on static, predefined partitions --- the tensor randomized extended block Kaczmarz (TREBK) method, its greedy counterpart (TREGBK), and the tensor randomized extended average block Kaczmarz (TREABK) method --- TGDBEK adapts the active block sizes dynamically at each step using a single intuitive threshold parameter 
$\eta$. We establish the theoretical linear convergence of TGDBEK to the unique minimum-norm least-squares solution $\fm{A}^\dagger * \fm{B}$. Extensive numerical benchmarks on synthetic dense and sparse tensor systems, as well as multidimensional multichannel color and 3D volumetric MRI image deblurring problems, demonstrate that TGDBEK substantially outperforms state-of-the-art tensor Kaczmarz solvers in both iteration count and CPU running time.
\end{abstract}



\section{Introduction}\label{sec:intro}

In this work, we focus on solving the tensor inconsistent linear system of the form
\begin{align}
\label{eq:eq1}
\mathcal{A} * \mathcal{X} =  \mathcal{B},\;  \fm B = \overline{\fm B} + \varepsilon
\end{align}
where $\mathcal{A}\in\mathbb{R}^{N_1\times N_2\times N_3},\;
\mathcal{B}\in\mathbb{R}^{N_1\times K\times N_3},\;
\mathcal{X}\in\mathbb{R}^{N_2\times K\times N_3},$  are third-order tensors, $\varepsilon \in \r^{N_1\times K\times N_3}$ the error tensor, and $*$ the t-product proposed by Kilmer and collaborators in \cite{kilmer2011facto,kilmer2013thirdOrder}. The t-product provides a linear algebraic paradigm for third-order tensors, which includes tensor analogues of matrix concepts like inverse of a tensor, transpose, orthogonal projections, and pseudo-inverse, etc.\ \cite{kilmer2011facto,kilmer2013thirdOrder}. Since its introduction in \cite{kilmer2011facto}, the t-product has been applied to several theoretical and practical applications. For instance, the authors in \cite{newman2018} proposed  a tensor neural network (t-NN) framework, a novel paradigm for designing deep neural networks.  We kindly refer to \cite{zhang2018,wang2020,qi2021,newman2018,miao2020gener} and to references therein for  applications of the t-product. Randomized algorithms built on this framework have proved particularly effective on imaging data: Tarzanagh and Michailidis \cite{tarzanagh2018} developed fast randomized algorithms for t-product tensor operations and decompositions, with applications to image and video recovery from incomplete and noisy data. The problem of the form \eqref{eq:eq1} arises in many practical applications, such as in color image recovery, MRI image deblurring with multiple channels, where the tensor $\fm X$ is the unknown clean tensor image to be found, $\fm A$ is a  tensor blur operator, and $\fm B$ is the observed blurred and noisy data. When the noise is relatively higher, it becomes difficult to solve this problem through direct methods. 

In the matrix setting, where $\fm A$ is a matrix, $\fm X$ and $\fm B$ are vectors, iterative methods like the randomized Kaczmarz (RK)\cite{strohmer2009}, a randomized variant of the classical Kaczmarz method \cite{kaczmarz1937angen}, can be used to solve the linear system ~\eqref{eq:eq1}. The RK method selects rows at each iteration with probability proportional to their Euclidean norm square, unlike the traditional Kaczmarz method, where rows are  selected cyclically and can result in a non-linear convergence.  In their seminal work, Strhomer et al. \cite{strohmer2009} proved that this row selection strategy makes the RK converge linearly in expectation for consistent and overdetermined linear systems. For inconsistent systems, however, RK only converges to within a fixed radius of the least-squares solution \cite{needell2010rando}. To solve inconsistent linear systems, Zouzias and Freris \cite{zouzias2013rando} extended the randomized Kaczmarz by introducing an additional iterative dual variable and proposed the randomized extended Kaczmarz (REK) which converges linearly in expectation to the minimum-norm least-squares solution $\m A^\dagger \m b$, where $\m A^\dagger$ is the Moore-Penrose pseudoinverse of the system matrix $\m A$. Explicitly, the REK method runs the RK iteration twice at each step: once on the consistent system $\m A^\top \m z = \m 0$ to update $\m z$, and once on $\m A \m x=\m b-\m z$ to update $\m x$, writing $\m z_0=\m b$,
\begin{align}
  \m z^{(k+1)} & = \m z^{(k)} - \frac{\big(\m A_{:j_k}\big)^\top\m z^{(k)}}{\|\m A_{:j_k}\|_2^2}\,\m A_{:j_k},    \label{eq:rek-z}\\
  \m x^{(k+1)} & = \m x^{(k)} + \frac{\m b_{i_k}-\m z^{(k+1)}_{i_k}-\m A_{i_k:}\m x^{(k)}}{\|\m A_{i_k:}\|_2^2}\,\big(\m A_{i_k:}\big)^\top   \label{eq:rek-x},
\end{align}
where $\m A_{i_k:}$ and $\m A_{:j_k}$ denote the $i_k$-th row and $j_k$-th column of $\m A$, $\m b_{i_k}$ and $\m z_{i_k}$ the corresponding entries of $\m b$ and of $\m z$, and the pair $(i_k,j_k)$ is drawn independently with probability proportional to $\Vert \m A_{i_k:}\Vert_2^2/\Vert \m A\Vert_F^2$ and $\Vert \m A_{:j_k}\Vert_2^2/\Vert \m A\Vert_F^2$, respectively. As $k\to\infty$, $\m z^{(k)}\to \m b_{\mathcal R(\m A)^\perp}$ and $\m x^{(k)}\to \m A^\dagger \m b$, each row and column of $\m A$ being used once at each iteration to update $\m x$ and $\m z$, respectively.

The randomized (extended) Kaczmarz methdod can be slow. Two practical ideas emerged to accelerate it. The first samples blocks of rows instead of a single row, and results in block variant methods and averaged-block variant methods\cite{du2020rando,tondji2023faster,necoara2019faste,needell2014paved,needell2015rando, moorman2021}. Averaging the block updates, as in \cite{du2020rando,tondji2023faster, necoara2019faste, moorman2021}, keeps the iteration pseudoinverse-free and amenable to parallel implementation. 
We  refer to  \cite{tondji2023accelerated, tondji2024acceleration, tondji2021linear} and an unified count of these accelerations and of their behaviour on inconsistent systems can be found in \cite{tondji2024thesis}. The second selects rows or blocks according to a greedy criterion where only those with largest residuals are chosen, resulting in greedy randomized Kaczmarz variants studied in \cite{bai2018greed,bai2019greed,bai2021greed} and in references therein.

Inconsistency can also be handled outside the extended framework, for example through Bregman projections. Sch\"opfer et al.\ \cite{schopfer2022extended} merged the extended and the sparse randomized Kaczmarz iterations and proved linear convergence to a \emph{sparse} least-squares solution, in a generalized form that also tolerates corruption in a few entries of the right-hand side. 
Tondji et al.\ \cite{tondji2024adaptive} assumed instead  an independent noisy sample is drawn at each step and showed that an adaptive stepsize, estimated from the data alone, reaches the \emph{exact} solution rather than a neighborhood of it; averaging the blocks and weighting them by their noise level accelerates this scheme further \cite{tondji2026accelerated}. 
Their mechanism for dealing with inconsistency differs from the REK method we build upon herein in this respect: they replace the Euclidean projection by a Bregman one --- and, sometimes in the block versions, the exact projection is replaced by an averaged-sum of projections onto each hyperplane composing the selected block --- and assume about the noise is redrawn each step, whereas the REK method and the method proposed below project out the part of the right-hand side lying outside the range of the operator. They are also stated for matrices only, and their t-product counterparts are still missing.

The block and greedy ideas we discussed previously have been recently brought to the t-product setting. Initially, Ma and Molitor \cite{molitor2022} proposed in their seminal work the tensor randomized Kaczmarz (TRK) method, which extends the RK method to solve tensor linear systems under the t-product. It was shown that the TRK method is equivalent to the block randomized Kaczmarz method in the Fourier domain and hence superior compared to the RK method applied to vectorized tensor systems. Subsequently, Chen and Qin \cite{chen2021} proposed the randomized regularized Kaczmarz method for tensor recovery. For applications of the randomized Kaczmarz method in tensor recovery and completion problems, please refer to \cite{chen2021,wang2022,du2021retr,wang2023,lia2024,zhang2024}. 
Castillo and collaborators \cite{castillo2025wisdm} further extend Kaczmarz- and Gauss--Seidel-type iterations to tensor regression under the t-product, with an application to image deblurring. The authors of \cite{ba02020TRB} introduced the tensor randomized block Kaczmarz (TRBK) method. To overcome inconsistency, Huang and Zhong \cite{huang2023TREK} extended REK directly to the t-product setting, resulting to the tensor randomized extended Kaczmarz (TREK) method. In fact, similar to the matrix randomized extended Kaczmarz, the tensor version TREK sets $\fm Z^{(0)}=\fm B$, and  at  each iteration $k$, draws the lateral and horizontal slice indexes $j$ and $i$  independently with probability  $\|\a{j}\|_F^2/\|\fm A\|_F^2$ and $\|\ai{i}\|_F^2/\|\fm A\|_F^2$ respectively, then updates the iterates using:
\begin{align}
  \fm Z^{(k+1)} &= \fm Z^{(k)} - \a{j}*\big((\a{j})^\top*\a{j}\big)^\dagger*(\a{j})^\top*\fm Z^{(k)}, \label{eq:trek-z}
  \\
  \fm X^{(k+1)} &= \fm X^{(k)} - (\ai{i})^\top*\big(\ai{i}*(\ai{i})^\top\big)^\dagger*\big(\ai{i}*\fm X^{(k)} - \fm B_{i,:,:} + \fm Z^{(k)}_{i,:,:}\big). \label{eq:trek-x}
\end{align}
Furthermore, in the same paper, Huang and Zhong \cite{huang2023TREK} also introduced two block variants of TREK, namely the tensor randomized extended block Kaczmarz (TREBK) and its greedy counterpart (TREGBK), which pre-partition the lateral and horizontal slices into fixed blocks rather than sampling a single slice at each step. Recently, An and collaborators \cite{lian2024TREABK} proposed tensor randomized extended average block Kaczmarz (TREABK) a pseudoinverse-free randomized extended average block Kaczmarz  which inherently uses the averaging idea appeared before in \cite{du2020rando,necoara2019faste}.

Many of these tensor block methods \cite{ba02020TRB,lian2024TREABK} we discussed previously rely primarily on static, predefined partitions of rows or columns of the tensor system, which of course can depend on the structure of the tensor system: sparse, dense, etc. In this paper, inspired by~\cite{su2025}, we apply the greedy randomized block extended Kaczmarz method's idea from matrices to tensor equations. Consequently, we propose a tensor greedy double block extended Kaczmarz (TGDBEK) method. The TGDBEK method constructs iteratively  lateral(columns) and horizontal(rows) slices at each iteration, and does not require a fixed partition of the tensor system. In fact, the TGDBEK method focuses projection computation on where the current residual is the largest. We prove the linear convergence  of the TGDBEK and demonstrate its practical effectiveness against TREK \cite{molitor2022}, TREBK \cite{huang2023TREK}, TREGBK \cite{huang2023TREK}, including applications to image deblurring and sparse tensor systems from SuiteSparse matrices \cite{suiteSparse}.

The remainder of the paper is organized as follows. In Section~\ref{sec:prelim}, we introduce the notations and definitions. In Section~\ref{sec:tdgbek}, we describe the TGDBEK algorithm and prove its linear convergence.  Numerical examples are provided in Section~\ref{sec:numexp} and finally we draw conclusions in Section~\ref{sec:concl}.

\section{Notations and preliminaries}\label{sec:prelim}

\subsection{Basic notation}

Throughout this paper, the calligraphic letters such as $\fm A, \fm B$ denote tensors. Bold capital letters represent matrices, for example, $\m A, \m B$. For an integer $m\geq 1$, we denote $[m]:=\{1, \ldots, m \}$. If  $\fm A$ is a third-order tensor, $\fm A_{ijk}$ denotes the $(i,j,k)$-th element. $\fm A_{i,:,:}$, $\fm A_{:,j,:}$, and $\fm A_{:,:,k}$ denote the horizontal, lateral, and frontal slices of $\fm A$, respectively. The tensor $\fm A$ and its three slices are illustrated in Fig.~\ref{fig:tensor-slices}. For $J\subseteq [N_1]$  we denote $\ai{J}\in \r^{\vert J\vert \times N_2\times N_3}$ a sub-tensor of $\fm A$ whose horizontal slice indexes belong to $J$. Similarly, for $U\subseteq [N_2]$ we denote $\a{U}\in \r^{N_1 \times \vert U \vert \times N_3}$ the sub-tensor $\fm A$ whose lateral slice indexes belong to $U$. The pseudo-inverse of $\fm A$ is denoted $\fm A^\dagger$ and the minimum and maximum nonzero singular values of a matrix are denoted $\sigma_{\min}^+(\cdot)$ and $\sigma_{\max}(\cdot)$, respectively. Similarly, we denote $\ran(\cdot)$ and $\ker(\cdot)$, respectively, the range and null spaces of a tensor $\fm A$ or a matrix $\m A$.

We now define the intrinsic geometric parameters that the main convergence theorem will rely on.
\subsubsection*{Definition of Intrinsic Geometric Parameters.}
Let $\Omega_r = \{ J \subseteq [N_1] : J \neq \emptyset \}$ be the collection of all non-empty row index subsets of $\mathcal{A}$. $J\in \Omega_r$ is an admissible row block.  We define the \textbf{intrinsic geometric parameters} of the tensor $\mathcal{A} \in \mathbb{R}^{N_1 \times N_2 \times N_3}$ as follows:

\begin{enumerate}
    \item \textbf{Tensor Block Condition Number:}
    \[
    \kappa_{\mathrm{block}}(\mathcal{A}) := \min_{J \in \Omega_r} \frac{\left(\sigma_{\min}^+(\mathrm{bcirc}(\mathcal{A}_{J,:,:}))\right)^2}{\sigma_{\max}^2(\mathrm{bcirc}(\mathcal{A}_{J,:,:}))} \in (0, 1]
    \]
    \textit{(measures the worst-case conditioning across all admissible  row blocks)}

    \item \textbf{Smallest Active Block Singular Value:}
    \[
    \sigma_{\min, \mathrm{block}}^+(\mathcal{A}) := \min_{J \in \Omega_r} \sigma_{\min}^+(\mathrm{bcirc}(\mathcal{A}_{J,:,:})) > 0
    \]

    \item \textbf{Subspace Error Fraction:}
    
    For any iteration $k \ge 0$, let $\theta_k$ denote the fraction of error captured by the active row range $R(\mathcal{A}_{J_k,:,:}^\top)$:
    \[
    \theta_k := \frac{\left\| \left(\mathcal{X}^{(k)} - \mathcal{X}_*\right)_{R\left((\mathcal{A}_{J_k,:,:})^\top\right)} \right\|_F^2}{\left\| \mathcal{X}^{(k)} - \mathcal{X}_* \right\|_F^2} \in (0, 1), \qquad \theta := \min_{k \ge 0} \theta_k > 0
    \]
\end{enumerate}

\begin{figure}[ht]
\centering
\centering
\begin{subfigure}[t]{0.23\linewidth}
  \centering
  \includegraphics[width=\linewidth]{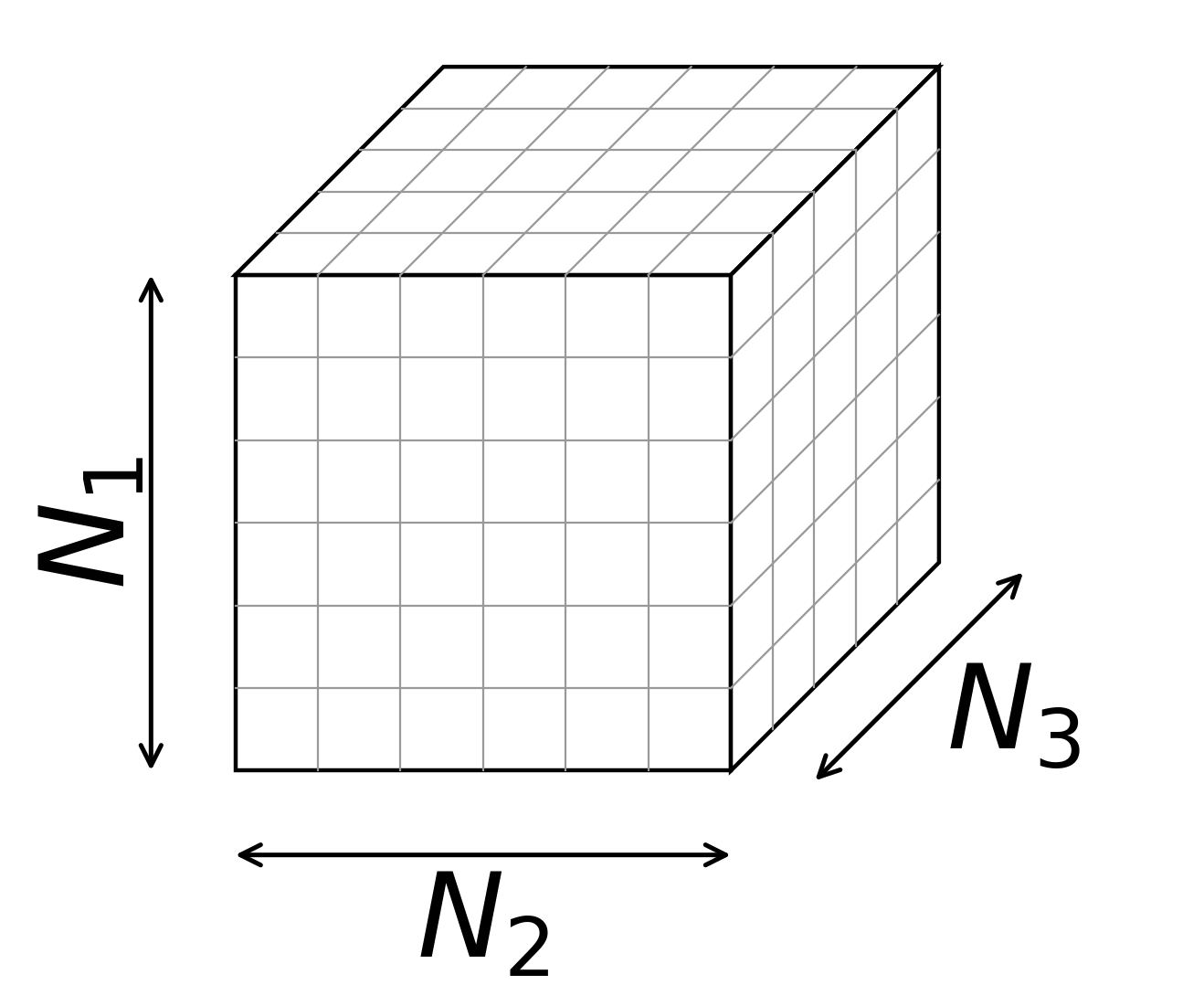}
  \caption{Tensor $\fm A$}
  \label{fig:tensor-cube}
\end{subfigure}\hfill
\begin{subfigure}[t]{0.23\linewidth}
  \centering
  \includegraphics[width=\linewidth]{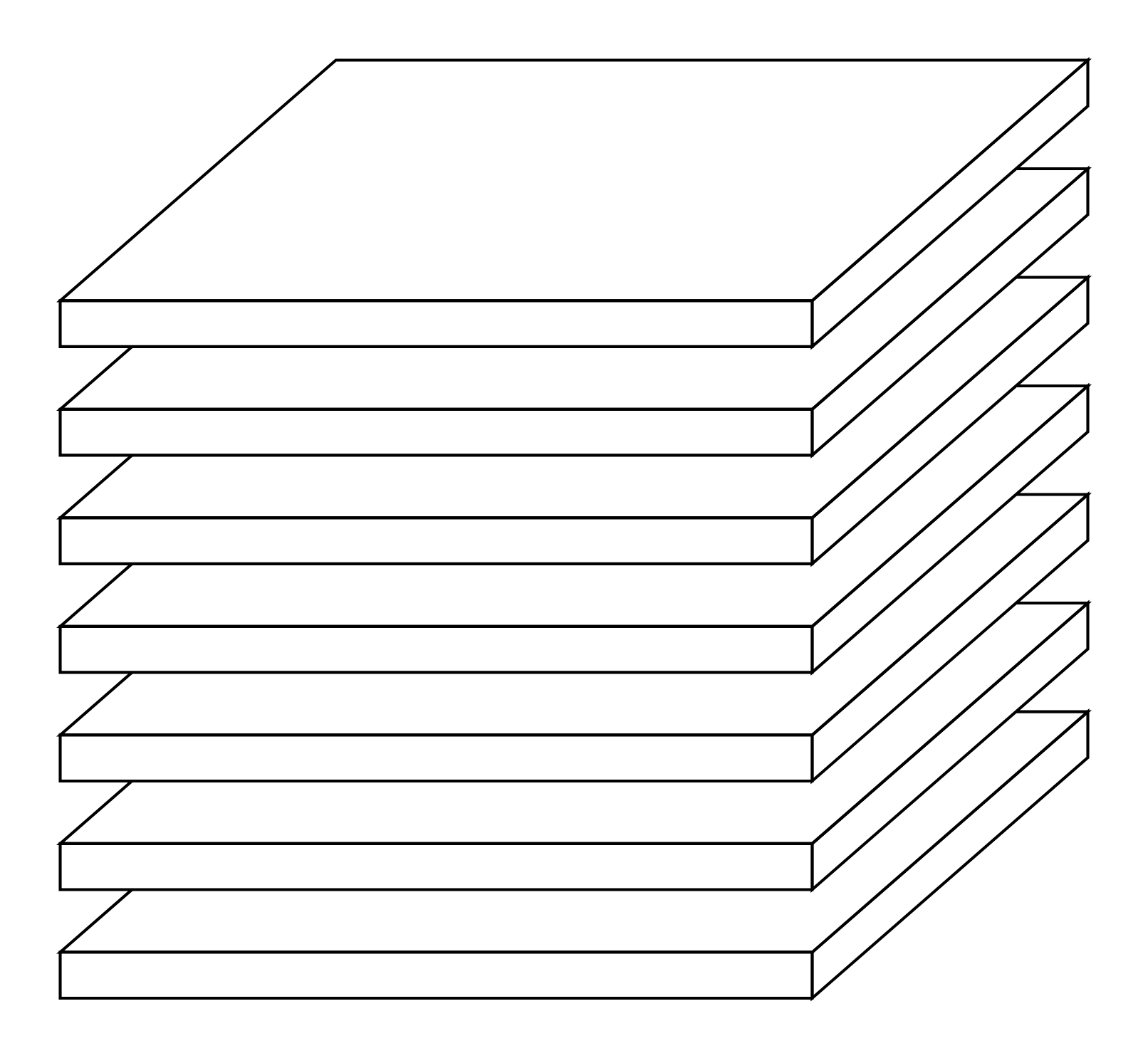}
  \caption{Horizontal slices $\fm A_{i,:,:}$}
  \label{fig:slice-horiz}
\end{subfigure}\hfill
\begin{subfigure}[t]{0.23\linewidth}
  \centering
  \includegraphics[width=\linewidth]{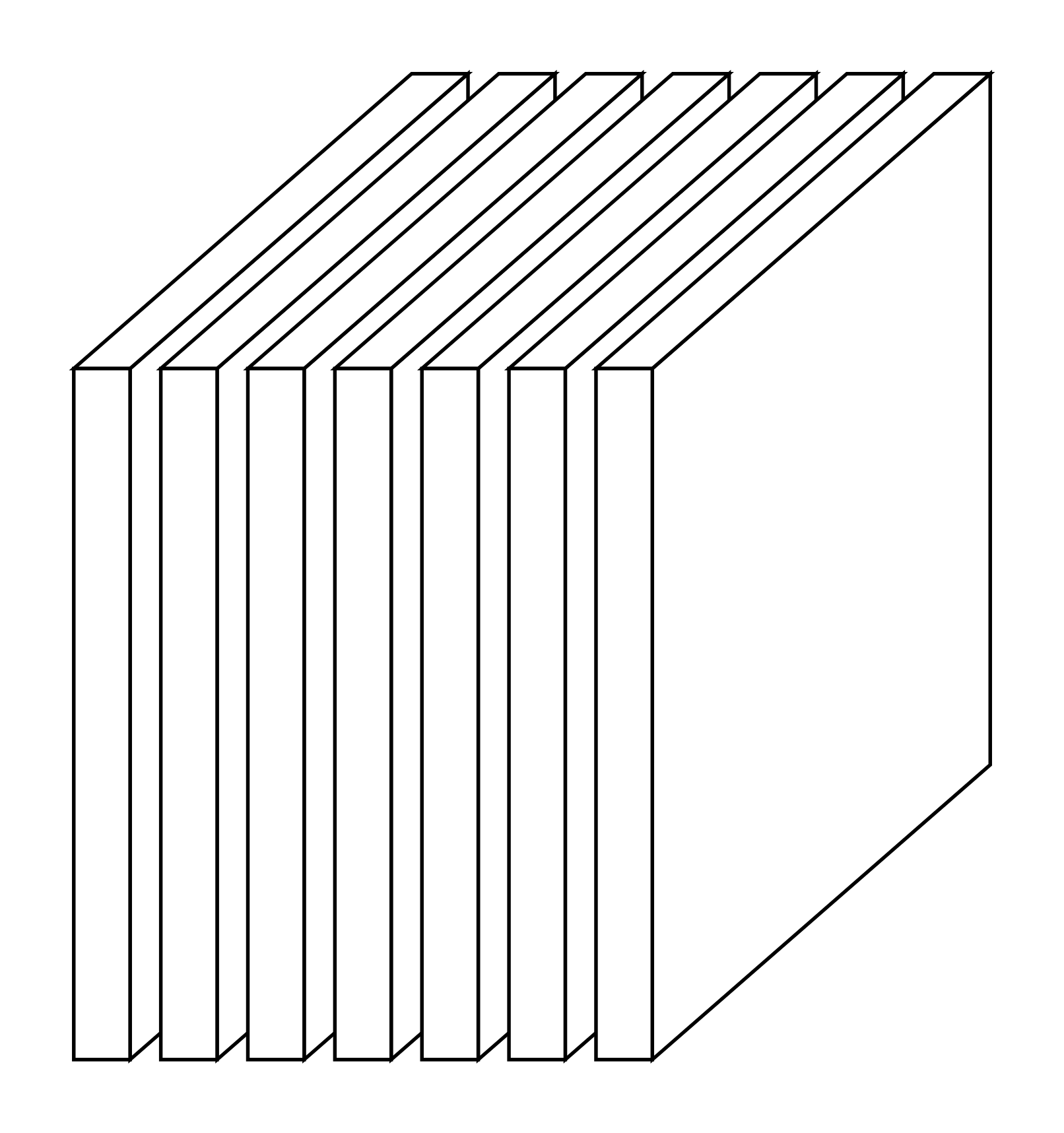}
  \caption{Lateral slices $\fm A_{:,j,:}$}
  \label{fig:slice-lat}
\end{subfigure}\hfill
\begin{subfigure}[t]{0.23\linewidth}
  \centering
  \includegraphics[width=\linewidth]{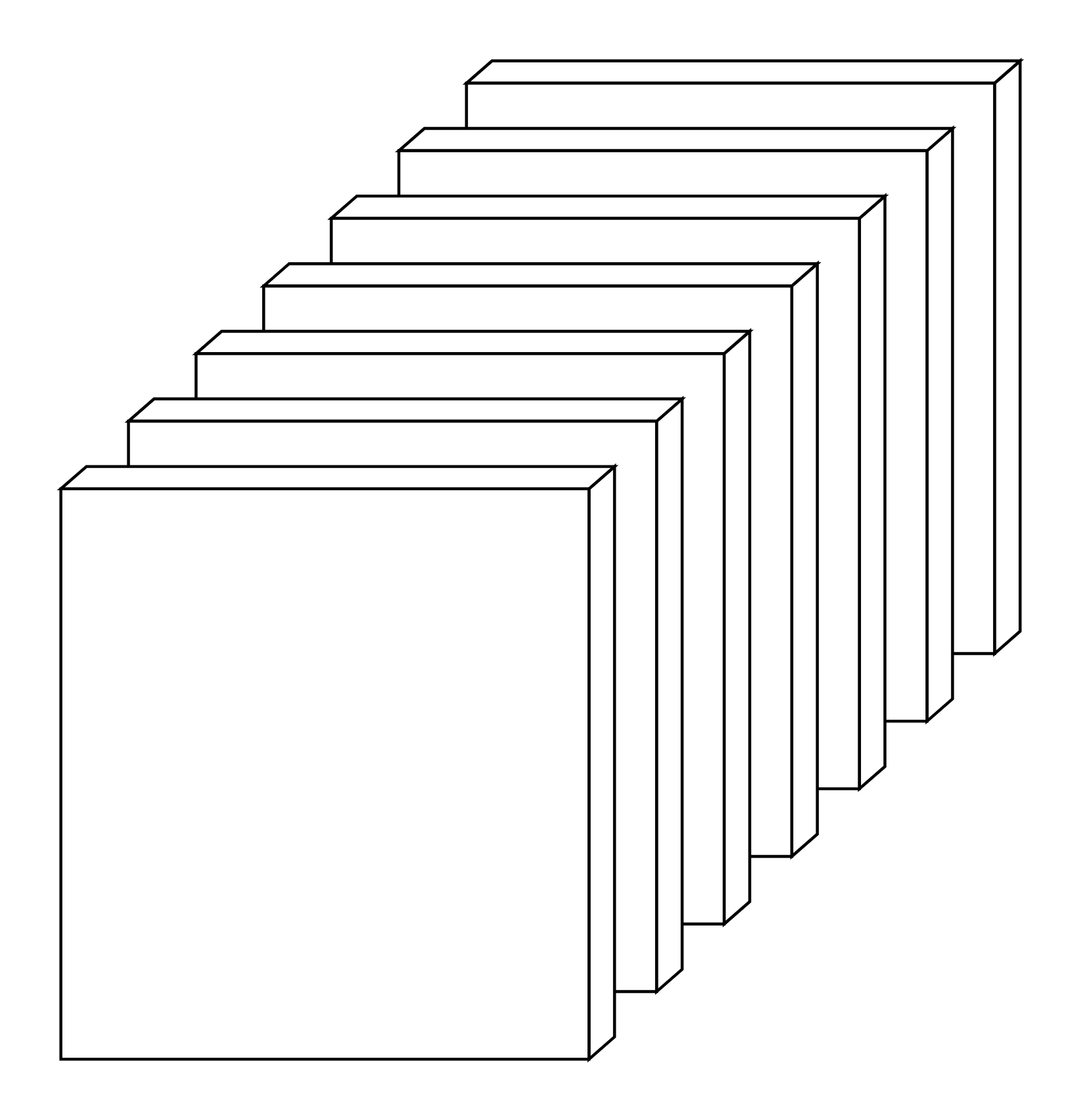}
  \caption{Frontal slices $\fm A_{:,:,k}$}
  \label{fig:slice-front}
\end{subfigure}
\caption{A third-order tensor $\fm A\in\mathbb{R}^{N_1\times N_2\times N_3}$ and its three  slices : horizontal (row), lateral (column), and frontal.}
\label{fig:tensor-slices}
\end{figure}

\subsection{Tensor preliminaries}
We follow Kilmer and Martin \cite{kilmer2011facto,kilmer2013thirdOrder} to provide the background for the t-product linear algebra.\\

Given a third-order tensor $\fm A\in \r^{N_1\times N_2\times N_3}$, we can unfold $\fm A$ into a matrix, and inversely fold the result back. The operation \textit{unfold} and its inverse \textit{fold} are given as
\begin{align*}
\unfold(\mcala):=\bem\fm A_{:,:,1}\\ \fm A_{:,:,2}\\ \vdots\\ \fm A_{:,:,N_3}\eem\in\mbbr^{N_1N_3\times N_2}, \qquad \fold\left(\bem\fm A_{:,:,1}\\ \fm A_{:,:,2}\\ \vdots\\ \fm A_{:,:,N_3}\eem\right):=\mcala.
\end{align*}

\noindent
The block circulant matrix of $\fm A$ is given as
\begin{align*}
    \bcirc(\mcala):=\bem\fm A_{:,:,1} &\fm A_{:,:,N_3} & \cdots & \fm A_{:,:,2}\\ \fm A_{:,:,2} & \fm A_{:,:,1} & \cdots &\fm A_{:,:,3}\\ \vdots &\vdots & \ddots &\vdots\\ \fm A_{:,:,N_3}& \fm A_{:,:,N_3-1}& \cdots &\fm A_{:,:,1}\eem\in\mbbr^{N_1N_3\times N_2N_3}.
\end{align*}

\begin{definition}[t-product \cite{kilmer2011facto,kilmer2013thirdOrder}]
For $\mcala\in\mbbr^{N_1\times N_2\times N_3}$ and $\mcalb\in\mbbr^{N_2\times K\times N_3}$, the t-product $\mcala*\mcalb$ is the $N_1\times K\times N_3$ tensor defined by
$$
\mcala*\mcalb:=\fold(\bcirc(\mcala)\unfold(\mcalb)).
$$
\end{definition}

\begin{definition}[\cite{kilmer2011facto}]
The identity tensor $\mcali\in\mbbr^{N\times N\times N_3}$ is the tensor whose first frontal slice is the $N\times N$ identity matrix, and whose other frontal slices are all zeros. The identity tensor satisfies $\fm A * \mcali = \mcali * \fm A = \fm A$ for all tensors $\fm A$ of compatible dimensions.
\end{definition}

\begin{definition}[\cite{kilmer2011facto,kilmer2013thirdOrder}]
The transpose of $\mcala\in\mbbr^{N_1\times N_2\times N_3}$, denoted $\mcala^\top$, is the $N_2\times N_1\times N_3$ tensor obtained by transposing each frontal slice and then reversing the order of transposed frontal slices $2$ through $N_3$.
\end{definition}

\begin{definition}[\cite{kilmer2011facto}]
A tensor $\fm A\in \r^{N_1\times N_2\times N_3}$ defines a linear map $\fm X \mapsto \fm A * \fm X$ from $\r^{N_2\times K\times N_3}$ to $\r^{N_1\times K\times N_3}$. Its range and null spaces are defined as
\begin{align*}
    \ran(\fm A) := \bigl\{\,\fm A * \fm Y \;:\; \fm Y\in \r^{N_2\times K\times N_3}\bigr\},\; \nul(\fm A) := \bigl\{\,\fm X \in \r^{N_2\times K\times N_3} \;:\; \fm A * \fm X = \fm O\,\bigr\}.
\end{align*}
\end{definition}

The following properties will be useful in the sequel.

\noindent
\textbf{Fact 1.} \cite{molitor2022} For $\fm A\in \r^{N_1\times N_2\times N_3}$ and $\fm B\in \r^{N_2\times K \times N_3}$, the block circulant operator satisfies:
\begin{align*}
    \bcirc(\fm A^\top) &= \bcirc(\fm A)^\top,\;
    \bcirc(\fm A * \fm B) = \bcirc(\fm A)\,\bcirc(\fm B).
\end{align*}
As an immediate consequence: $(\fm A * \fm B)^\top = \fm B^\top * \fm A^\top$.

\begin{definition}[\cite{du2021retr, jin2017generalized}]\label{def:pseudoinverse}
For $\fm A\in \r^{N_1\times N_2\times N_3}$, its Moore--Penrose pseudoinverse, denoted
$\fm A^\dagger$, is the unique $N_2\times N_1\times N_3$ tensor that  satisfies
$\bcirc(\fm A^\dagger) = \bcirc(\fm A)^\dagger$.
\end{definition}


\noindent
\textbf{Fact 2.} \cite[Theorem~3.1]{jin2017generalized} For $\fm A\in \r^{N_1\times N_2\times N_3}$,
\begin{align*}\label{eq:penrose}
    \fm A * \fm A^\dagger * \fm A = \fm A, \;
    \fm A^\dagger * \fm A * \fm A^\dagger = \fm A^\dagger, \;
    \left(\fm A * \fm A^\dagger\right)^\top = \fm A * \fm A^\dagger, \;
    \left(\fm A^\dagger * \fm A\right)^\top = \fm A^\dagger * \fm A.
\end{align*}
These are the tensor analogues of the matrix Penrose conditions; they follow from the matrix Penrose conditions applied to $\bcirc(\fm A)$, together with Definition~\ref{def:pseudoinverse}, Fact~1, and the injectivity of $\bcirc(\cdot)$.\\

\noindent
\textbf{Fact 3.} For $\fm A\in \r^{N_1\times N_2\times N_3}$, the following holds:
\begin{align*}
(\fm A ^\dagger)^\top = (\fm A^\top)^\dagger, \;\; \fm A^\top * \fm A * \fm A^\dagger = \fm A^\top.
\end{align*}
\begin{proof}
The first assertion holds from Fact~1 and  the Definition~\ref{def:pseudoinverse}. For the second, it suffices to use  the Fact~2 and the first assertion to obtain  
\begin{align*}
    \fm A^\top * \fm A * \fm A ^\dagger = \fm A^\top * (\fm A * \fm A^\dagger)^\top = \fm A^\top *(\fm A^\dagger)^\top * \fm A^\top = \fm A^\top 
\end{align*}
\end{proof}

\begin{definition}[\cite{kilmer2011facto}]\label{orth-proj-def}
A tensor $\fm P$ is a projector if $\fm P * \fm P = \fm P$, and it is orthogonal if $\fm P^\top = \fm P$.
\end{definition}

\begin{lemma}[\cite{kilmer2011facto}]\label{lem:lem-proj}
For $\fm A\in \r^{N_1\times N_2 \times N_3}$, the tensors $\fm P=\fm A * \fm A^\dagger$ and $\fm Q = \fm A^\dagger * \fm A$ are orthogonal projection tensors.
\end{lemma}
The proof of this Lemma follows directly from Facts 1 \& 2, and the Definition ~ \ref{orth-proj-def}.\\

Since Lemma~\ref{lem:lem-proj} holds for any tensor $\fm A\in\r^{N_1\times N_2\times N_3}$,
it applies in particular to the sub-tensors used in the TGDBEK algorithm in Section ~\ref{sec:tdgbek}. More precisely, for any index set $U_k\subseteq[N_2]$, the tensor $\left(\fm A^\top_{:,U_k,:}\right)^\dagger * \fm A^\top_{:,U_k,:}$ is an orthogonal projector. Similarly, for any index set $J_k\subseteq[N_1]$, the tensor
$\fm A_{J_k,:,:} * \left(\fm A_{J_k,:,:}\right)^\dagger$ is also an orthogonal projector.\\


\begin{definition}[\cite{kilmer2011facto}]  For $\fm A \in \r^{N_1 \times N_2 \times N_3}$, its spectral norm and the Frobenius norm are defined as follows
\begin{align}
    \Vert \fm A \Vert_2  := \Vert \bcirc (\fm A ) \Vert_2, \qquad 
    \Vert \fm A \Vert_F := \sqrt{\langle \fm A , \fm A \rangle} = \sqrt{\sum_{i,j,k}\fm A_{ijk}^2}
\end{align}
\end{definition}
From this definition, we see that $\sigma_{\max}(\fm A):= \sigma_{\max}(\bcirc(\fm A))$ and  similarly, ${\sigma_{\min}^+}(\fm A) := {\sigma_{\min}^{+}}(\bcirc(\fm A ))$.  The t-inner product is the sum of all $A_{ijk}$ squared. The relation between $\Vert \bcirc(\fm A)\Vert_F$ and $\Vert \fm A\Vert_F$ is given as 
\begin{align}
    \Vert \bcirc(\fm A ) \Vert_F = \sqrt{N_3} \Vert \fm A \Vert_F \label{eq:bcirc-frob}
\end{align}

\begin{lemma}[\cite{chen2021,du2021retr}]\label{lem:lem-12}
For $\fm A\in \r^{N_1\times N_2\times N_3}$ and any $\fm X\in \r^{N_2\times K\times N_3}$, we have
\begin{align}
    \Vert \fm A * \fm X \Vert_F^2 \leq \sigma_{\max}^2\!\left(\bcirc(\fm A)\right)\Vert \fm X \Vert_F^2. \label{eq:sigma-max-bound}
\end{align}
Moreover, for $\fm X \in \ran(\fm A^\top) = \nul(\fm A)^{\perp}$ 
\begin{align}
    \Vert \fm A * \fm X \Vert_F^2 \geq {\sigma_{\min}^+}^2\!\left(\bcirc(\fm A)\right)\Vert \fm X \Vert_F^2. \label{eq:sigma-min-bound}
\end{align}
\end{lemma}

\begin{lemma}[Tensor Pythagorean theorem]\label{lem:app-pythag}
Let $\fm P\in\r^{N\times N\times N_3}$ be a t-orthogonal projector. Then for any tensor $\fm X \in \r^{N\times K\times N_3}$ we have,
\begin{align}
    \Vert \fm X \Vert_F^2 = \Vert \fm P * \fm X \Vert_F^2 + \Vert (\mcali - \fm P) * \fm X \Vert_F^2. \label{thm-pythagore}
\end{align}
\end{lemma}
\begin{proof}
Since $\fm P$ is a t-orthogonal projector, $(\mcali - \fm P)$ is also a t-orthogonal projector with $\fm P * (\mcali - \fm P) = \fm O$, so their ranges are orthogonal. Therefore $\langle \fm P * \fm X,\; (\mcali - \fm P)*\fm X \rangle = 0$, and the identity follows from $\fm X = \fm P*\fm X + (\mcali-\fm P)*\fm X$.
\end{proof}

\begin{lemma}\label{lem:app-range-facts}
Let $\fm A\in\r^{N_1\times N_2\times N_3}$ and $\fm X_\ast = \fm A^\dagger * \fm B$. Then:
\begin{enumerate}
    \item $\ran(\fm A^\dagger) = \ran(\fm A^\top)$ and $\ran((\fm A^\dagger)^\top) = \ran(\fm A)$
    \item $\fm X_\ast\in \ran(\fm A^\top)$ is the unique minimum-norm least-squares solution to $\fm A * \fm X = \fm B$
    \item $\ran(\fm A^\top)\cap\nul(\fm A) = \{\fm O\} $
\end{enumerate}
\end{lemma}

\section{The TGDBEK algorithm}\label{sec:tdgbek}

In this section, we construct the tensor greedy double extended Kaczmarz, as illustrated in Algorithm~\ref{alg:alg-1}. We assume the system \eqref{eq:eq1} is inconsistent, thus $\fm B \notin \ran(\fm A)$. From the orthogonal decomposition  $\ran(\fm A) \perp \nul(\fm A^\top)$ there must exist $\fm Z \in \nul(\fm A^\top)$ such that  $\fm B = (\fm B - \fm Z) + \fm Z$ and
$\fm B - \fm Z \in \ran(\fm A)$, in order the system  $\fm A * \fm X = \fm B - \fm Z$ be consistent. Solving \eqref{eq:eq1} is thus equivalent to finding $\fm Z$ and $\fm X$
that satisfy 
\begin{align*}
  \fm A^\top * \fm Z = \fm O \qquad \text{and} \qquad \fm A * \fm X = \fm B - \fm Z,
\end{align*}
which obviously is a tensor analogue of the extended Kaczmarz decomposition of Zouzias and
Freris~\cite{zouzias2013rando}. The resulting $\fm X$ is the minimum-norm
least-squares solution $\fm X_\ast = \fm A^\dagger * \fm B$ of eq.~\eqref{eq:eq1}.

The TREK Algorithm (eqs.~\eqref{eq:trek-z}--\eqref{eq:trek-x})  solves the two problems above by applying the tensor analogue of RK  twice  at each iteration: once on $\fm A^\top * \fm Z = \fm O$, the iterate  $\fm Z^{(k)}$ is updated along a single, randomly drawn lateral slice index, and once on
$\fm A * \fm X = \fm B - \fm Z$, and the iterate $\fm X^{(k)}$ is updated along a single, randomly drawn horizontal slice index. Each slice of $\fm A$ is thus used once at each iteration to update $\fm Z$ and $\fm X$, exactly as in REK style ~\eqref{eq:rek-z}--\eqref{eq:rek-x}. 
Following Sun and Qin~\cite{su2025}, we can construct blocks of lateral(column) and horizontal(row) slice indexes using a greedy strategy, instead of simply sampling individual slices. The greedy  deterministically selects, at the $k$-th iteration, a block of lateral slice set $U_k$ and a block of horizontal slice set  $J_k$ of $\fm A$, then update the $\fm X$ and $Z$ iterates slightly as in ~\eqref{eq:trek-z}--\eqref{eq:trek-x}. In fact, the current iterates are orthogonally projected onto the solution spaces of $\fm A_{:,U_k,:}^\top * \fm Z = \fm O$ and
$\fm A_{J_k,:,:} * \fm X^{(k)} = \fm B_{J_k,:,:} - \fm Z^{(k+1)}_{J_k,:,:}$. The sets $U_k$ and $J_k$ are non-empty by design and are given in the Algorithm~\ref{alg:alg-1} below. By selecting the threshold $\eta$, we control the sizes of $U_k$ and $J_k$. Precisely, the objective here is to identify the largest entries of the residual tensors $-\fm A^\top * \fm Z^{(k)}$ and $ - \fm A * \fm X^{(k)} +  \fm Z^{(k+1)} - \fm B $ at each iteration, ensure that these entries are prioritized for
elimination and so to accelerate the convergence of $\fm Z^{(k)}$ and $\fm X^{(k)}$. Whereas TREBK or TREGBK instead update a static, predefined block of slices  and randomly selects them at each
step, TGDBEK builds these blocks at each iteration from the current residual and  are not necessarily randomly sampled.

\refstepcounter{algorithm}%
\label{alg:alg-1}%
\begin{center}
\begin{tabular*}{0.95\textwidth}{l}
\toprule {\bf Algorithm 1:} Tensor Greedy Double Block Extended Kaczmarz (TGDBEK)\\
\hline \noalign{\smallskip}
\quad {\bf Input}: $\mathcal{A}\in\mathbb{R}^{N_1\times N_2\times N_3}$, $\fm B\in\mathbb{R}^{N_1\times K\times N_3}$, iterations $\ell$, parameter $\eta\in(0,1]$.\\
\noalign{\smallskip}
\quad {\bf Initialize}: $\mathcal{X}^{(0)}=\mathcal{O}$ and $\mathcal{Z}^{(0)}=\fm B$.\\
\noalign{\smallskip}
\quad {\bf for} $k=0,1,\ldots,\ell-1$ {\bf do}\\
\noalign{\smallskip}
\quad \qquad $\displaystyle \varepsilon_k^{z}=\eta\max_{1\le j\le N_2}\left\{\frac{\big\|\mathcal{A}_{:,j,:}{}^\top * \mathcal{Z}^{(k)}\big\|_{\mathrm{F}}^2}{\big\|\mathcal{A}_{:,j,:}\big\|_{\mathrm{F}}^2}\right\}$\\ \noalign{\smallskip}
\quad \qquad $\displaystyle U_k=\left\{\, j\in[N_2] : \big\|\mathcal{A}_{:,j,:}{}^\top * \mathcal{Z}^{(k)}\big\|_{\mathrm{F}}^2 \ge \varepsilon_k^{z} \big\|\mathcal{A}_{:,j,:}\big\|_{\mathrm{F}}^2 \right\}$\\ \noalign{\smallskip}
\quad \qquad Update $\displaystyle \mathcal{Z}^{(k+1)}=\mathcal{Z}^{(k)}-\mathcal{A}_{:,U_k,:}*(\mathcal{A}_{:,U_k,:})^\dagger*\mathcal{Z}^{(k)}$\\
\noalign{\smallskip}
\quad \qquad $\displaystyle \varepsilon_k^{x}=\eta\max_{1\le i\le N_1}\left\{\frac{\big\|\fm B_{i, :, :}-\mathcal{Z}^{(k+1)}_{i, :, :}-\mathcal{A}_{i, :, :}*\mathcal{X}^{(k)}\big\|_{\mathrm{F}}^2}{\big\|\mathcal{A}_{i,:,:}\big\|_{\mathrm{F}}^2}\right\}$\\ \noalign{\smallskip}
\quad \qquad $\displaystyle J_k=\{\, i\in[N_1] : \big\|\fm B_{i, :, :}-\mathcal{Z}^{(k+1)}_{i, :, :}-\mathcal{A}_{i, :, :}*\mathcal{X}^{(k)}\big\|_{\mathrm{F}}^2\ge \varepsilon_k^{x}\,\Vert \ai{i} \Vert_F^2\}$\\ \noalign{\smallskip}
\quad \qquad Update $\displaystyle
\mathcal{X}^{(k+1)}=\mathcal{X}^{(k)}+(\fm A_{J_k, :, :})^\dagger*
\Big(\fm B_{J_k,:,:}-\mathcal{Z}^{(k+1)}_{J_k,:,:}-\fm A_{J_k, :, :}*\mathcal{X}^{(k)}\Big)$\\
\noalign{\smallskip}
\quad {\bf end}\\
\bottomrule
\end{tabular*}
\end{center}

The Theorem~\ref{thm:main} states the lower bound for the sequence $\{\fm X^{(k)}\}$ generated by the Algorithm~\ref{alg:alg-1} and proves that it converges to the minimum-norm solution of eq.~\eqref{eq:eq1}. Its proof uses the  Lemma~\ref{lem:z-step}  which establishes the convergence of the sequence $\{\fm Z^{(k)}\}$ generated by the Algorithm~\ref{alg:alg-1}. 

\begin{lemma}[Linear Convergence of the $\fm Z$-Iterates]\label{lem:z-step}
Let
$\fm B^\perp := \fm B - \fm A * \fm A^\dagger * \fm B$ denote the component of $\fm B$ orthogonal to $\ran(\fm A)$. The sequence $\{\fm Z^{(k)}\}_{k \geq 0}$ generated by Algorithm~\ref{alg:alg-1} converges linearly to $\fm B^\perp$ satisfying, for all $k\geq 0$:
\begin{align}
    \Vert\fm Z^{(k+1)} - \fm B^\perp\Vert_F^2 \leq \beta\,\Vert\fm Z^{(k)} - \fm B^\perp\Vert_F^2,
\end{align}
with contraction factor: $\beta = 1 - \eta \frac{{\sigma_{\min}^+}^2(\bcirc(\fm A))}{\Vert \bcirc(\fm A) \Vert_F^2} \in (0,1)$.
\end{lemma}

\begin{proof}[Proof of Lemma~\ref{lem:z-step}]
From Fact~3 we obtain
\begin{align*}
    \fm A^\top *\fm B^\perp  =\fm A^\top  *\fm B - \fm A^\top * \fm A * \fm A^\dagger * \fm B =\fm A ^\top * \fm B  - \fm A ^\top * \fm B = 0.
\end{align*}
In particular, for any $U_k \subset [N_2]$ we get $\fm A_{:, U_k, :}^\top *\fm B^\perp = \fm O$. Now, using the fact $\fm A * \fm A^\dagger = \left(\fm A * \fm A^\dagger \right)^\top = (\fm A^\dagger)^\top * \fm A ^\top = (\fm A^\top )^\dagger * \fm A^\top$
(which follows from Penrose condition~(iii) and Fact~3) and the $\fm Z-$ update in Algorithm~\ref{alg:alg-1}, we can get
\begin{align*}
    \fm Z^{(k+1)} - \fm B^\perp  & = \left(\fm Z^{(k)} - \fm B^\perp\right) - \left(\fm A^\top_{:, {U_k}, :} \right)^\dagger* \fm A^\top_{:, U_k, :} *\left( \fm Z ^{(k)}-\fm B^\perp \right)\\
    & = \left(\fm I - \left( \a{U_k}^\top\right) ^\dagger * \a{U_k}^\top\right)*\left( \fm Z^{(k)}- \fm B^\perp \right).
\end{align*}

Since $\left(\fm A^\top_{:,U_k,:}\right)^\dagger \fm A^\top_{:,U_k,:}$ is a t-orthogonal projector, the Pythagorean theorem (Lemma~\ref{lem:app-pythag}) gives
\begin{align}
    \Vert\fm Z^{(k+1)} - \fm B^\perp\Vert_F^2  = \Vert\fm Z^{(k)} - \fm B^\perp\Vert_F^2 - \Vert\left(\fm A^\top_{:,U_k,:}\right)^\dagger* \fm A^\top_{:, U_k, :}*\left(\fm Z^{(k)}-\fm B^\perp\right)\Vert_F^2. \label{eq:eq-1b}
\end{align}

We need to upper bound the  second term of the addition above. Since $\a{U_k}^\top * \fm Z^{(k)} \in \ran(\a{U_k}^\top) = \ran\!\left(((\a{U_k}^\top)^\dagger)^\top\right)$, the lower singular value bound in Lemma~\ref{lem:lem-12} applied to $(\a{U_k}^\top)^\dagger$ gives
\begin{align}
    \Vert \left(\a{U_k}^\top \right)^\dagger * \a{U_k}^\top *\left(\fm Z^{(k)}-\fm B^\perp \right) \Vert_F^2 & =  \Vert \left(\a{U_k}^\top \right)^\dagger * \a{U_k}^\top *\fm Z^{(k)}\Vert_F^2 \notag\\
    & \geq {\sigma_{\min}^+}^2((\a{U_k}^\top) ^\dagger)\Vert \a{U_k}^\top * \fm Z^{(k)} \Vert_F^2
  \quad \text{(by Lemma \ref{lem:lem-12})} \notag\\
    & = {\sigma_{\min}^+}^2(\bcirc(\a{U_k}^\top) ^\dagger) \sum_{j\in U_k} \Vert \a j ^\top * \fm Z^{(k)} \Vert_F^2 \notag\\
    & \geq {\sigma_{\min}^+}^2(\bcirc(\a{U_k}^\top) ^\dagger) \epsilon^z_k \Vert \a{U_k} \Vert_F^2 \notag\\
    & = \frac{\Vert \a{U_k} \Vert_F^2}{\sigma_{\max}^2(\bcirc(\a{U_k}))} \epsilon_k^z \notag\\
    & \geq \frac{1}{N_3} \epsilon_k^z,  \label{eq:eq-2}
\end{align}
where in the last inequality we use eq.~\eqref{eq:bcirc-frob}, and the facts $\sigma_{\max}^2(\m M^\top) = \sigma_{\max}^2(\m M)$ and $\Vert \m M \Vert_F^2 \geq \Vert \m M \Vert_2$, for any matrix $\m M$. 
Now, observe that
\begin{align*}
    \a{U_{k-1}}^\top *\fm Z^{(k)} & = \a{U_{k-1}}^\top *\left( \fm Z^{(k-1)} - \a{U_{k-1}}*\a{U_{k-1}}^\dagger * \fm Z^{(k-1)} \right)\\
    & =  \a{U_{k-1}}^\top*\fm Z^{(k-1)} - \a{U_{k-1}}^\top *\a{U_{k-1}} * \a{U_{k-1}}^\dagger* \fm Z^{(k-1)} = \fm O,
\end{align*}
so that for all $k=0, 1, \ldots$,
\begin{align*}
    \Vert \fm A^\top * \fm Z^{(k)}\Vert_F^2  & = \sum_{j \in U_{k-1}} \Vert \a j ^\top  * \fm Z^{(k)} \Vert_F^2 + \sum_{j\in U^c_{k-1}}\Vert \a j ^\top  * \fm Z^{(k)} \Vert_F^2 \\
    &=  \sum_{j\in U^c_{k-1}}\Vert \a j ^\top  * \fm Z^{(k)} \Vert_F^2  =  \sum_{j\in U^c_{k-1}} \frac{ \Vert \a j ^\top  * \fm Z^{(k)} \Vert_F^2}{ \Vert \a j \Vert_F^2}\Vert \a j \Vert_F^2\\
    & \leq \max_{j \in [N_2]}\left\{  \frac{ \Vert \a j ^\top  * \fm Z^{(k)} \Vert_F^2}{ \Vert \a j \Vert_F^2} \right\}\left(\Vert \fm A \Vert_F^2 - \Vert \a{U_k} \Vert_F^2 \right)\\
    & \leq \max_{j \in [N_2]}\left\{  \frac{ \Vert \a j ^\top  * \fm Z^{(k)} \Vert_F^2}{ \Vert \a j \Vert_F^2} \right\} \Vert \fm A \Vert_F^2, 
\end{align*}
since  $\Vert \fm A \Vert_F^2 - \Vert \a{U_k} \Vert_F^2\leq \Vert \fm A \Vert_F^2$. Hence
\begin{align}
    \epsilon^z_k = \eta \max_{j \in [N_2]}\left\{  \frac{ \Vert \a j ^\top  * \fm Z^{(k)} \Vert_F^2}{ \Vert \a j \Vert_F^2} \right\} &\geq \eta \frac{ \Vert \fm A^\top * \fm Z^{(k)}\Vert_F^2}{\Vert \fm A \Vert_F^2 } \notag\\
    & \geq \eta\frac{ {\sigma_{\min}^+}^2(\bcirc(\fm A^\top))}{\Vert \fm A \Vert_F^2} \Vert \fm Z^{(k)}-\fm B^\perp \Vert_F^2 \notag\\
    & = \eta N_3 \frac{ {\sigma_{\min}^+}^2(\bcirc(\fm A^\top))}{\Vert \bcirc(\fm A )\Vert_F^2} \Vert \fm Z^{(k)}-\fm B^\perp \Vert_F^2. \label{eq:eq-3}
\end{align}
The last inequality holds from  Lemma~\eqref{lem:lem-12} applied to $\fm A^\top$ and again eq.~\eqref{eq:eq-3} uses the relation between $\Vert \fm A \Vert$ and $\Vert \bcirc(\fm A) \Vert_F$ given in eq.~\eqref{eq:bcirc-frob}.
Substituting \eqref{eq:eq-3} into \eqref{eq:eq-2}, we obtain the lower bound
\begin{align}
     \Vert \left(\a{U_k}^\top \right)^\dagger * \a{U_k}^\top *\left(\fm Z^{(k)}-\fm B^\perp \right) \Vert_F^2 & \geq \eta  \frac{{\sigma_{\min}^+}^2(\bcirc(\fm A^\top))}{\Vert \bcirc(\fm A^\top)\Vert_F^2} \Vert \fm Z^{(k)}-\fm B^\perp \Vert_F^2. \label{eq:eq-4}
\end{align}
Finally, the combination of  \eqref{eq:eq-4} and \eqref{eq:eq-1b} gives
\begin{align*}
     \Vert \fm Z^{(k+1)} - \fm B^\perp \Vert_F^2
     \leq \underbrace{\left( 1 - \eta \frac{{\sigma_{\min}^+}^2(\bcirc(\fm A^\top))}{\Vert \bcirc(\fm A ) \Vert_F^2} \right)}_{=:\,\beta}\Vert \fm Z^{(k)} - \fm B^\perp \Vert_F^2. 
\end{align*}
\end{proof}

\begin{theorem}[Linear Convergence of TGDBEK]
\label{thm:main}
Assume the system \eqref{eq:eq1} is inconsistent and let $\mathcal{B}^\perp := \mathcal{B} - \mathcal{A} * \mathcal{A}^\dagger * \mathcal{B} \in \mathrm{null}(\mathcal{A}^\top)$ denote the component of $\mathcal{B}$ orthogonal to the range space $R(\mathcal{A})$. 

Let $J_k \subseteq [N_1]$ be the active row blocks selected by Algorithm~\ref{alg:alg-1} at iteration $k$. Then, the sequence $\{\mathcal{X}^{(k)}\}_{k \ge 0}$ generated by Algorithm~\ref{alg:alg-1} converges linearly in the Frobenius norm to the unique minimum-norm least-squares solution $\mathcal{X}_* = \mathcal{A}^\dagger * \mathcal{B}$. 

Specifically, for all $k \ge 0$:
\begin{equation}
\label{eq:rate}
\|\mathcal{X}^{(k+1)} - \mathcal{X}_*\|_F^2 \le \alpha^{k+1} \|\mathcal{X}^{(0)} - \mathcal{X}_*\|_F^2 + \gamma \left( \frac{\alpha^{k+1} - \beta^{k+1}}{\alpha - \beta} \right) \beta \|\mathcal{B} - \mathcal{B}^\perp\|_F^2 \quad (\text{if } \alpha \neq \beta),
\end{equation}
and if $\alpha = \beta$:
\begin{equation}
\label{eq:main_error_bound_equal}
\|\mathcal{X}^{(k+1)} - \mathcal{X}_*\|_F^2 \le \alpha^{k+1} \|\mathcal{X}^{(0)} - \mathcal{X}_*\|_F^2 + \gamma (k+1) \alpha^{k+1} \|\mathcal{B} - \mathcal{B}^\perp\|_F^2,
\end{equation}
where
\begin{equation*}
\label{eq:rate_constants}
\alpha = 1 - \theta \cdot \kappa_{\mathrm{block}}(\mathcal{A}) \in (0, 1), \quad  \gamma = \frac{1}{\left(\sigma_{\mathrm{min, block}}^+(\mathcal{A})\right)^2}, \quad \beta = 1 - \eta\,
     \frac{{\sigma_{\min}^+}^2\!\big(\bcirc(\fm A)\big)}
         {\Vert \bcirc(\fm A)\Vert_F^2} \in (0, 1),
\end{equation*}
with $\eta \in (0,1]$ the parameter of Algorithm~\ref{alg:alg-1}.
\end{theorem}

The proof of this theorem relies on the previous  contraction lemma \ref{lem:z-step}.

\begin{proof}[Proof of the main Theorem~\ref{thm:main}.]
Since $\b{J_k} = \ai{J_k}*\fm X^\ast + \bp{J_k}$, the $\fm X$-update becomes
\begin{align*}
   \fm X^{(k+1)} = \x{k} - \ai{J_k}^\dagger * \left(\ai{J_k}*\left(\x{k} - \fm X^\ast\right) + \left(\z{k+1}{J_k} - \bp{J_k}\right)\right).
\end{align*}
Subtracting $\fm X_\ast$ from both sides yields
\begin{align}
   \x{k+1} - \fm X_\ast = \left(\fm I - \ai{J_k}^\dagger * \ai{J_k}\right)*\left(\x{k} - \fm X^\ast\right) + \ai{J_k}^\dagger*\left(\z{k+1}{J_k} - \bp{J_k}\right).\label{eq:x-decomp}
\end{align}
$\ai{J_k}^\dagger * \ai{J_k}$ is a t-orthogonal projector  by Lemma~\ref{lem:lem-proj} ,  so the range $\ran\left(\fm I - \ai{J_k}^\dagger * \ai{J_k}\right)$ is orthogonal to $\ran(\ai{J_k}^\dagger)$. The Pythagorean theorem (Lemma~\ref{lem:app-pythag}) gives
\begin{align}
    \Vert\x{k+1} - \fm X_\ast\Vert_F^2 = \Vert\left(\fm I - \ai{J_k}^\dagger * \ai{J_k}\right)*\left(\x{k} - \fm X^\ast\right)\Vert_F^2 + \Vert\ai{J_k}^\dagger*\left(\z{k+1}{J_k} - \bp{J_k}\right)\Vert_F^2. \label{eq:eq-5}
\end{align}
It suffices to upper bound the two terms of the right-hand side in equation above. For the first term, note that we can decompose $\fm X^{(k+1)} - \fm X_\ast $ as follow
\begin{align}
    \x{k}-\fm X_\ast = \left( \x{k} - \fm X_\ast \right)_{\ran((\ai{J_k})^\top)} + \left(\x k - \fm X_\ast \right)_{\ran((\ai{J_k})^\top)^\perp }\label{eq:x-decomp-sub-space}.
\end{align}
Let  recall $\theta_k \in (0,1]$ defined by 
\begin{align}
    \theta_k = \frac{\bigl\Vert(\x{k} - \fm X_\ast)_{\ran((\ai{J_k})^\top)}\bigr\Vert_F^2}{\Vert\x{k} - \fm X_\ast\Vert_F^2},\label{eq:theta-def}
\end{align}
The constant $\theta_k$ cannot be zero. In fact, $\theta_k = 0$ means its numerator is also zero, which implies  $\left( \x{k} - \fm X_\ast \right)_{\ran((\ai{J_k})^\top)}= 0$ or $\x{k+1} - \fm X_\ast \in \ker(\ai{J_k}))=\ran((\ai{J_k})^\top)^\perp$, i.e. $\ai{J_k} \left( \x{k+1} - \fm X_\ast \right) = 0$. Hence, from the decomposition in eq~\eqref{eq:x-decomp}, it results $\x{k+1} - \fm X_\ast = \ai{J_k}^\dagger*\left(\z{k+1}{J_k} - \bp{J_k}\right)= \ai{J_k}^\dagger*\z{k+1}{J_k}$. In particular, $\x{k+1} - \fm X_\ast \in \ran((\ai{J_k})^\dagger) = \ran((\ai{J_k})^\top)$. We have both $\x{k+1} - \fm X_\ast \in \ran((\ai{J_k})^\top)$ and in its orthogonal, i.e. $\x{k+1} - \fm X_\ast  = 0$, for $k$.\\
\noindent
The first term in  the right-hand side of \eqref{eq:eq-5} satisfies
\begin{align}
\Vert \left( \fm I - \ai{J_k}^\dagger *  \ai{J_k} \right) *\left(\x k - \fm X^\ast   \right) \Vert_F^2  & = \Vert \x k - \fm X_\ast \Vert_F^2 \label{eq:just-1} \\
&- \Vert  \ai{J_k}^\dagger *  \ai{J_k}  *\left(\x k - \fm X^\ast   \right) \Vert_F^2  \notag\\
& \leq  \Vert \x k - \fm X_\ast \Vert_F^2 - {\sigma_{\min}^+}^2(\ai{J_k}^\dagger) \label{eq:just-2} \\
&\times \Vert \ai{J_k}  *\left(\x k - \fm X^\ast   \right) \Vert_F^2 \notag\\
& \leq  \Vert \x k - \fm X_\ast \Vert_F^2  -{\sigma_{\min}^+}^2(\ai{J_k}^\dagger) {\sigma_{\min}^+}^2(\ai{J_k})\notag \\
&\times \Vert \left( \x{k} - \fm X_\ast \right)_{\ran((\ai{J_k})^\top)} \Vert_F \label{eq:proof-1}\\
& = \Vert \x k - \fm X_\ast \Vert_F^2  - \frac{{\sigma_{\min}^+}^2(\bcirc(\ai{J_k}))}{\sigma^2_{\max}(\bcirc(\ai{J_k}))}\notag \\
& \times \theta_k \Vert \x k - \fm X_\ast \Vert_F^2  \label{eq:proof-2} \\
& = \left( 1 -  \theta_k \frac{{\sigma_{\min}^+}^2(\bcirc(\ai{J_k}))}{\sigma^2_{\max}(\bcirc(\ai{J_k}))}\right) \Vert \x k - \fm X_\ast \Vert_F^2 \notag  \\
& \leq \left( 1 - \theta \cdot  \kappa_{\mathrm{block}}(\mathcal{A})\right) \Vert \x k - \fm X_\ast \Vert_F^2 \label{eq:eq-6},
\end{align}
where the eq.~\eqref{eq:just-1} holds from Lemma~\ref{lem:app-pythag}, eq.~ \eqref{eq:just-2} uses  Lemma~\ref{lem:lem-12} part two and since by design, the greedy row blocks $J_k$ generated at each iteration  by Algorithm~\ref{alg:alg-1}  are non-empty, we have $J_k\in \Omega_r$. To see eq.~\eqref{eq:proof-1}, with the help of  Lemma~\ref{lem:app-range-facts} we get  $\ran((\ai{J_k})^\top)^\perp = \ker(\ai{J_k})$, then applying  $ \ai{J_k} $ in eq.~\eqref{eq:x-decomp-sub-space} gives
\begin{align*}
    \ai{J_k} \left(\x k - \fm X^\ast   \right)  
     = \ai{J_k} \left( \x{k} - \fm X_\ast \right)_{\ran((\ai{J_k})^\top)}.
\end{align*}
Since, $\left( \x{k} - \fm X_\ast \right)_{\ran((\ai{J_k})^\top)}\in \ran((\ai{J_k})^\top)$, apply Lemma~\ref{lem:lem-12} to get eq.~\eqref{eq:proof-1}, i.e.
\begin{align*}
    \Vert  \ai{J_k} \left(\x k - \fm X^\ast   \right)   \Vert_F^2 
    &= \Vert \ai{J_k} \left( \x{k} - \fm X_\ast \right)_{\ran((\ai{J_k})^\top)} \Vert \\
    & \geq {\sigma_{\min}^+}^2(\ai{J_k}) \Vert \left( \x{k} - \fm X_\ast \right)_{\ran((\ai{J_k})^\top)} \Vert_F^2.
\end{align*}
Furthermore, eq.~ \eqref{eq:proof-2} uses  ${\sigma_{\min}^+}^2(\m M^\dagger) = \frac{1}{\sigma_{\max}^2(\m M)}$ which holds for any matrix $\m M$.

\noindent
Using Lemma~\ref{lem:lem-12}, the second term in the right-hand side  of \eqref{eq:eq-5} satisfies,
\begin{align}
  \Vert  \ai{J_k}^\dagger* \left(\z{k+1}{J_k} - \bp{J_k}  \right) \Vert_F^2  & \leq \sigma_{\max}^2(\bcirc((\ai{J_k})^\dagger) )\Vert  \z{k+1}{J_k} - \bp{J_k} \Vert_F^2 \notag\\
  & = \frac{1}{{\sigma_{\min}^+}^2(\bcirc(\ai{J_k}))} \Vert  \z{k+1}{J_k} - \bp{J_k} \Vert_F^2 \notag\\
  & \leq \frac{1}{{\sigma_{\min}^+}^2(\bcirc(\ai{J_k}))} \Vert  \fm Z^{(k+1)} - \fm B^\perp \Vert_F^2 .
  \label{eq:eq-7}
\end{align}
It results from  \eqref{eq:eq-5}, \eqref{eq:eq-6}, and \eqref{eq:eq-7} that 
\begin{align*}
 \Vert \x{k+1} - \fm X_\ast \Vert_F^2 &  \leq \left( 1 -  \theta\frac{{\sigma_{\min}^+}^2(\bcirc(\ai{J}))}{\sigma^2_{\max}(\bcirc(\ai{J}))}\right) \Vert \x k - \fm X_\ast \Vert_F^2  \\
 & +  \frac{1}{{\sigma_{\min}^+}^2(\bcirc(\ai{J_k}))} \Vert  \fm Z^{(k+1)} - \fm B^\perp \Vert_F^2\\
 & \leq \left( 1 - \theta\frac{{\sigma_{\min}^+}^2(\bcirc(\ai{J}))}{\sigma^2_{\max}(\bcirc(\ai{J}))}\right) \Vert \x k - \fm X_\ast \Vert_F^2 \\
 & + \frac{\beta}{{\sigma_{\min}^+}^2(\bcirc(\ai{J_k}))} \Vert  \fm Z^{(k)} - \fm B^\perp \Vert_F^2\\
 & \leq \left( 1 - \theta\frac{{\sigma_{\min}^+}^2(\bcirc(\ai{J}))}{\sigma^2_{\max}(\bcirc(\ai{J}))}\right) \Vert \x k - \fm X_\ast \Vert_F^2 \\
 & + \frac{\beta}{(\sigma_{\min, \mathrm{block}}^+(\mathcal{A}))^2} \Vert  \fm Z^{(k)} - \fm B^\perp \Vert_F^2.
\end{align*}
Now, let $\alpha =  1 -\theta\cdot  \kappa_{\mathrm{block}}(\mathcal{A})$ and $\gamma \;=\; \frac{1}{(\sigma_{\min, \mathrm{block}}^+(\mathcal{A}))^2}$.
Then we have
\begin{align*}
    \Vert \x{k+1} - \fm X_\ast \Vert_F^2  & \leq \alpha  \Vert \x k - \fm X_\ast \Vert_F^2 + \gamma\beta \Vert \fm Z^{(k)} - \fm B^\perp \Vert_F^2\\
    & \leq \alpha \left( \alpha  \Vert \x k - \fm X_\ast \Vert_F^2 + \gamma\beta \Vert \fm Z^{(k)} - \fm B^\perp \Vert_F^2 \right) + \gamma\beta^2 \Vert \fm Z^{(k-1)} - \fm B^\perp \Vert_F^2\\
    &\leq \cdots\\
    & \leq \alpha^{k+1} \Vert \x 0 - \fm X_\ast \Vert_F^2  + \gamma\sum_{l=0}^k \alpha^{l} \beta^{k-l+1}  \Vert \fm Z^{(k-l)} - \fm B^\perp \Vert_F^2.
\end{align*}
The map $m\mapsto\Vert \fm Z^{(m)}-\fm B^\perp\Vert _F^2$ is a contraction by Lemma~\ref{lem:z-step}.
Therefore $\Vert \fm Z^{(m)}-\fm B^\perp\Vert _F^2\leq\Vert \fm Z^{(0)}-\fm B^\perp\Vert _F^2$
for all $m\geq 0$.
With the initialization $\fm Z^{(0)}=\fm B$ we have
$\Vert \fm Z^{(0)}-\fm B^\perp\Vert _F^2=\Vert \fm B-\fm B^\perp\Vert _F^2$,
so every term $\Vert \fm Z^{(k-l)}-\fm B^\perp\Vert _F^2$ in the sum is bounded above by
$\Vert \fm B - \fm B^\perp \Vert_F^2$.
Hence
\begin{align*}
    \Vert \x{k+1} - \fm X_\ast \Vert_F^2
    &\leq \alpha^{k+1} \Vert \x 0 - \fm X_\ast \Vert_F^2  + \gamma\Vert \fm B - \fm B^\perp \Vert_F^2\sum_{l=0}^k \alpha^{l} \beta^{k-l+1}\\
    &\leq \alpha^{k+1} \Vert \x 0 - \fm X_\ast \Vert_F^2 + \gamma \left(\frac{\alpha^{k+1} -\beta^{k+1}}{\alpha - \beta}\right) \beta\,\Vert \fm B - \fm B^\perp \Vert_F^2
\end{align*}
\end{proof}

\begin{remark}
We cannot exclude $\alpha = \beta$ \textit{a priori}, in which case the  quotient $\frac{\alpha^{k+1} -\beta^{k+1}}{\alpha - \beta}  = \sum_{l=0}^k \alpha^{l} \beta^{k-l+1}$ and the noise component of \eqref{eq:rate} becomes  clearly
\begin{align*}
    \gamma (k+1)\alpha^{k+1} \Vert \fm B - \fm B^\perp \Vert_F^2,
\end{align*}
which obviously tends to zero  as $k$ increases (by the ratio test) since generally $\alpha < 1$ . When $N_3=1$, $\fm A\in\mathbb R^{N_1\times N_2\times 1}$ is exactly a matrix $A\in\mathbb R^{N_1\times N_2}$, the t-product $*$ reduces to the ordinary
matrix multiplication, and $\bcirc(\fm A)=A$, so also
$\Vert\bcirc(\fm A)\Vert_F=\Vert \fm A\Vert_F$.
Algorithm~\ref{alg:alg-1} then reduces to the GDBEK algorithm of
\cite{su2025}, and the theorem above recovers 
\cite[Theorem~2.1]{su2025}: the only difference occurs in Lemma~\ref{lem:z-step},  where  we use $\Vert \fm A \Vert_F^2 - \Vert \a{U_k} \Vert_F^2\leq \Vert \fm A \Vert_F^2$.
\end{remark}

\section{Numerical experiments}\label{sec:numexp}
In this section, we conduct numerical experiments to demonstrate the effectiveness
of TGDBEK method and compare it with TREK \cite{molitor2022}, TREBK  ~\cite{huang2023TREK}, and TREGBK ~\cite{huang2023TREK} methods. All the experiments were implemented in Python using PyTorch and run on an Apple MacBook Pro with an M3 chip. The complete code is publicly available at \url{https://github.com/jnlandu/tensor-greedy-double-extended-kaczmarz}.

For a fair comparison, we use the number of iteration steps ("IT") and the running time ("CPU"). Here, IT and CPU represent the arithmetic mean of the number of iterations and running time required to  each algorithm to reach the tolerance $10^{-5}$ over $N$ independent trials. Sometimes we take $N=5$, $N=10$ or $N=20$. The algorithm terminates when the relative squared error (RSE) satisfies
\begin{align*}
    \mathrm{RSE} := \frac{\Vert \fm X^{(k)} - \fm X_\ast\Vert_F^2}{\Vert \fm X_\ast\Vert_F^2} < 10^{-5},
\end{align*}
where $\fm X_\ast = \fm A^\dagger * \fm B$ is generated using PyTorch, or when the total number of iterations  is exceeded. We set the initial iterates $\fm X^{(0)} = \fm O$ and $\fm Z^{(0)} = \fm B$. Finally, the right-hand side $\fm B$ is obtained as  $\fm B= \fm A * \fm X_\ast + \fm \epsilon$, where the the noise $\epsilon$, it is given as follows $\epsilon = a  \zeta  \frac{\Vert\fm B\Vert_F} {\Vert \zeta \Vert_F} $, with $a$ a positive constant and $\zeta \sim \fm N(0,1)$ a gaussian noise.

\begin{example}[\textbf{Dense}]\label{ex:dense}
In this example, we apply  TREK, TREBK,  TREGBK  and TGDBEK to solve a randomly generated dense tensor systems. The system tensor $\fm A \in \mathbb{R}^{500\times n\times 10}$ is constructed with  i.i.d standard Gaussian entries, and  similarly the  exact solution $\fm X_\ast \in \mathbb{R}^{n\times 10\times 10}$. We take  $n\in\{20,30,40,50,60,80\}$.
The right-hand side is computed as $\fm B = \fm A \ast \fm X_\ast + \fm\epsilon$, where the noise $\fm\epsilon$ is generated as described above with a noise level $a=10^{-3}$.
For TGDBEK  we choose the threshold $\eta=0.6$. The  block-partitioned methods use sequential blocks of size $\tau=10$, defined by $I_k= \{(k-1)\tau + 1, \ldots, \min\{k\tau, N\}\}$ for $k=1, \ldots,  \lfloor N/\tau \rfloor$, applied separately to $N_1 = 500$ and $N_2 =n$,  resulting to $10$ row blocks and  $\lfloor n/10 \rfloor$ column blocks. For this example, the total number of iteration is  $800$ and result averaged over 5 trials. See Table~\ref{tab:ex1} and Figs.~\ref{fig:ex1}--\ref{fig:ex1-z-rse}.

\begin{table}[htbp]
\caption{IT and CPU (averaged over 5 trials) for dense overdetermined tensor systems $\fm A\in\mathbb{R}^{500\times n\times 10}$, noise $a=10^{-3}$, maximum $800$ iterations. Bold: best per $n$.}\label{tab:ex1}
\centering\small\setlength{\tabcolsep}{4pt}
\begin{tabular}{llrrrrrr}
\toprule
Method & Metric & $n=20$ & $n=30$ & $n=40$ & $n=50$ & $n=60$ & $n=80$ \\
\midrule
\multirow{2}{*}{TREK}   & IT     & 547   & 800   & 800   & 800   & 800   & 800   \\
& CPU(s) & 1.114 & 1.458 & 1.569 & 1.569 & 1.473 & 1.559 \\
\midrule
\multirow{2}{*}{TREBK}  & IT     & 36    & 69    & 97    & 122   & 160   & 247   \\
& CPU(s) & 0.705 & 1.415 & 2.478 & 2.849 & 3.760 & 5.835 \\
\midrule
\multirow{2}{*}{TREGBK} & IT     & 37    & 63    & 90  & 123   & 151   & 217   \\
& CPU(s) & 0.580 & 1.168 & 1.860 & 2.707 & 3.416 & 5.632 \\
\midrule
\multirow{2}{*}{\textbf{TGDBEK}} & IT     & \textbf{23} & \textbf{28} & \textbf{33} & \textbf{35} & \textbf{37} & \textbf{44} \\
& CPU(s) & \textbf{0.497} & \textbf{0.652} & \textbf{1.007} & \textbf{1.401} & \textbf{1.460} & \textbf{2.385} \\
\bottomrule
\end{tabular}
\end{table}

From Table~\ref{tab:ex1}, as for IT and CPU, it is clear that the TGDBEK method requires the fewest iterations and least running time across all values of $n$. For example, the TREK method exhausts the total iteration without even reaching the given tolerance for $n\geq 30$. At $n=80$, TGDBEK converges in only $44$ iterations, compared to $247$ for TREBK and $217$ for TREGBK, this is undoubtedly a considerable reduction in both IT and CPU.

As can be seen from Fig.~\ref{fig:ex1-rse}, TGDBEK reaches $\mathrm{RSE} < 10^{-5}$ in $23$ iterations at $n=20$ and $44$ iterations at $n=80$. 
Fig.~\ref{fig:ex1-rse-cpu} shows the convergence curves (RSE) against CPU. Again, we can see that TGDBEK achieves the given tolerance faster than all competing methods. At $n=80$, TREK exhausted fast the total number of iterations without achieving the prescribed tolerance.
Fig.~\ref{fig:ex1} reports IT and CPU across all the values of $n$. It can be seen that  TGDBEK has the  least iteration count, and  requires fewer CPU among all the  methods.  Fig.~\ref{fig:ex1-z-rse} plots the errors for both  the x and z components against for $n=20$ and $n=80$. It reveals the convergence of the  sequence $\{\fm Z^{(k)}\}$. Finally, Fig.~\ref{fig:ex1-blocks} reports the cardinalities of greedy blocks $U_k$ and $J_k$ per iteration, and on average the TGDBEK  method builds blocks of size $4$ and $37.9$ for $n=20$ and $12.8$ and $40$ for $n=80$, respectively for $U_k$ and $J_k$.
\end{example}

\begin{figure}[!ht]
\centering
\begin{subfigure}[t]{0.48\linewidth}
  \includegraphics[width=\linewidth]{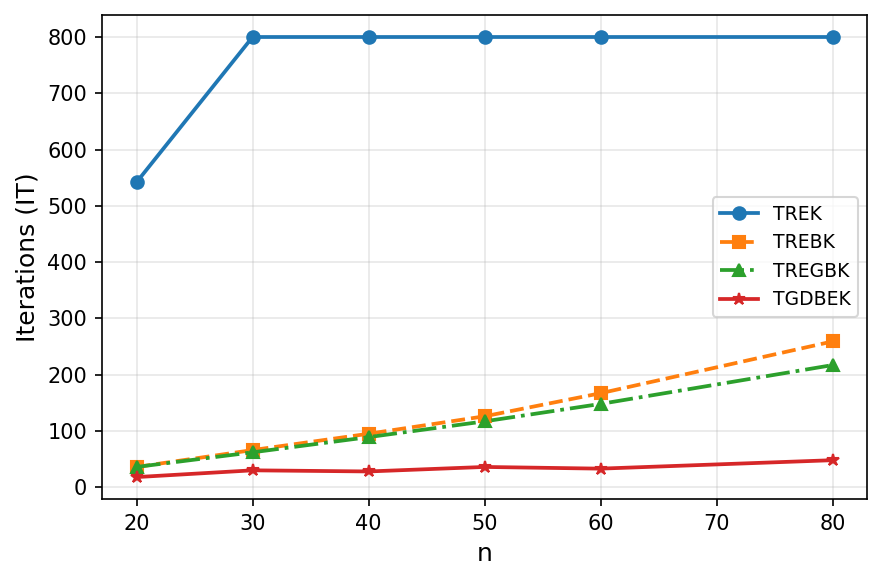}
  \caption{Iteration count (IT) vs.\ $n$.}
  \label{fig:ex1-IT}
\end{subfigure}\hfill
\begin{subfigure}[t]{0.48\linewidth}
  \includegraphics[width=\linewidth]{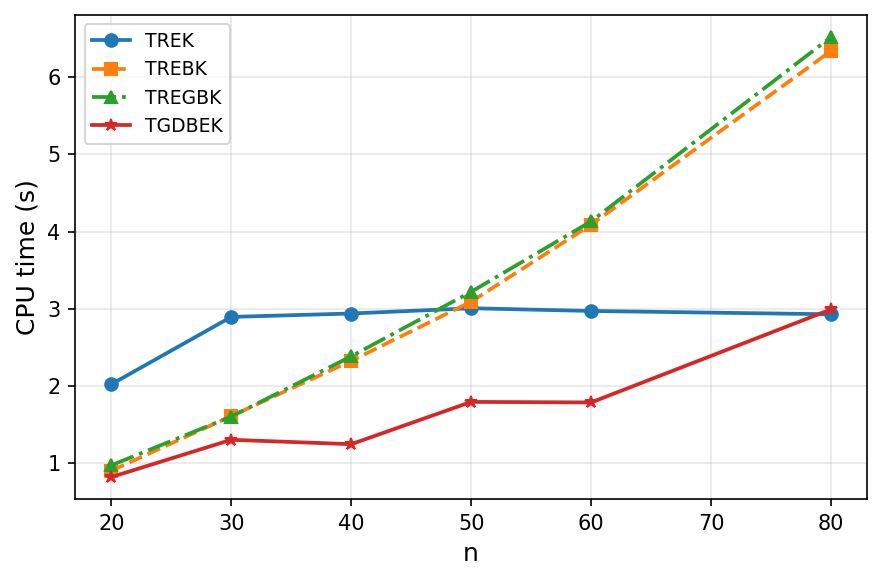}
  \caption{CPU time (s) vs.\ $n$.}
  \label{fig:ex1-CPU}
\end{subfigure}
\caption{IT and CPU as functions of $n$ for dense overdetermined systems ($\fm A\in\mathbb{R}^{500\times n\times 10}$, noise $a=10^{-3}$, max $800$ iterations, 5 trials). TREK exhausts the budget for $n\geq 30$.}\label{fig:ex1}
\end{figure}

\begin{figure}[!ht]
\centering
\includegraphics[width=\linewidth]{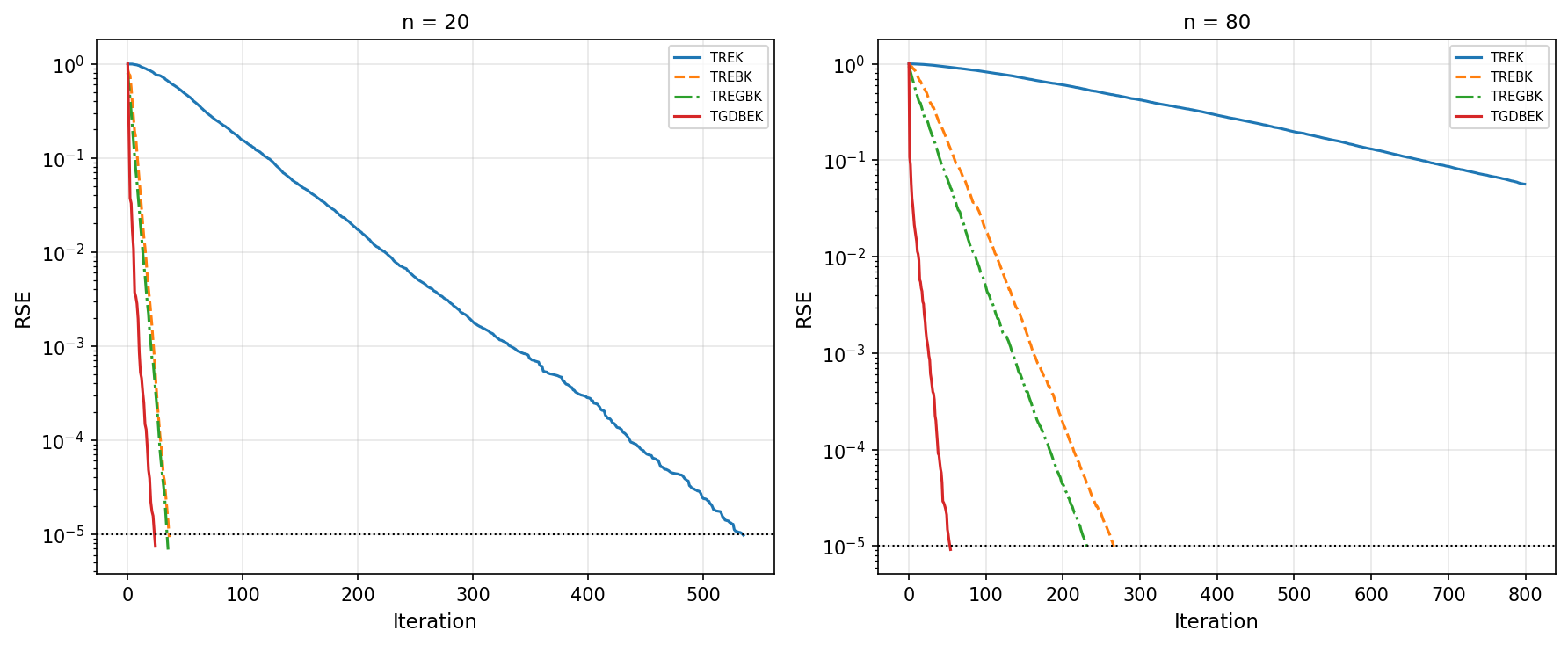}
\caption{RSE versus iteration for $n=20$ (left) and $n=80$ (right). TGDBEK converges in 23 and 43 iterations respectively; TREK exhausts the 800-iteration budget at $n=80$. Dotted line: tolerance $\mathrm{tol}=10^{-5}$.}\label{fig:ex1-rse}
\end{figure}

\begin{figure}[!ht] 
\centering
\includegraphics[width=\linewidth]{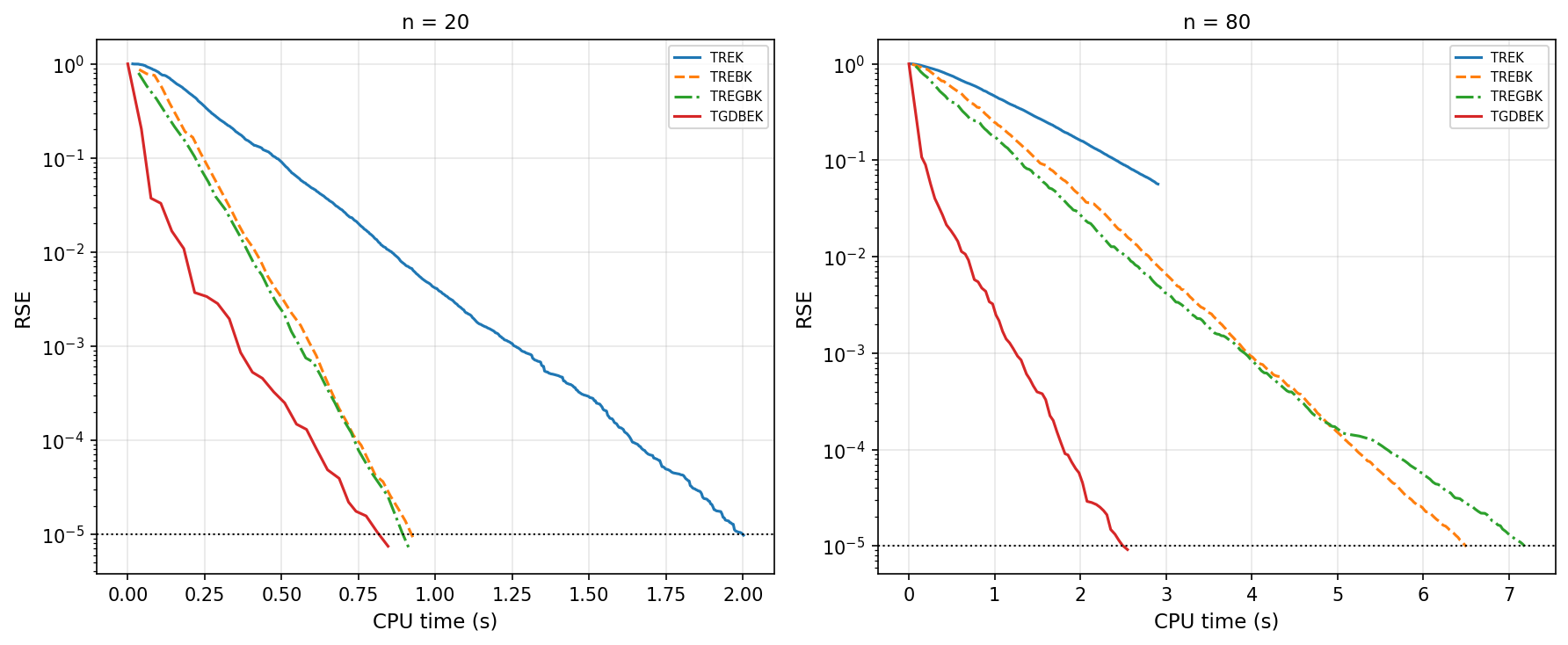}
\caption{$\mathrm{RSE}$ vs CPU  for $n=20$ (left) and $n=80$ (right). TGDBEK reaches the tolerance in the least time across both problem sizes; TREK exhausts the budget at $n=80$ without converging.}\label{fig:ex1-rse-cpu}
\vspace{-0.3cm}
\end{figure}

\begin{figure}[!ht]
\centering
\begin{subfigure}[t]{0.48\linewidth}
  \includegraphics[width=\linewidth]{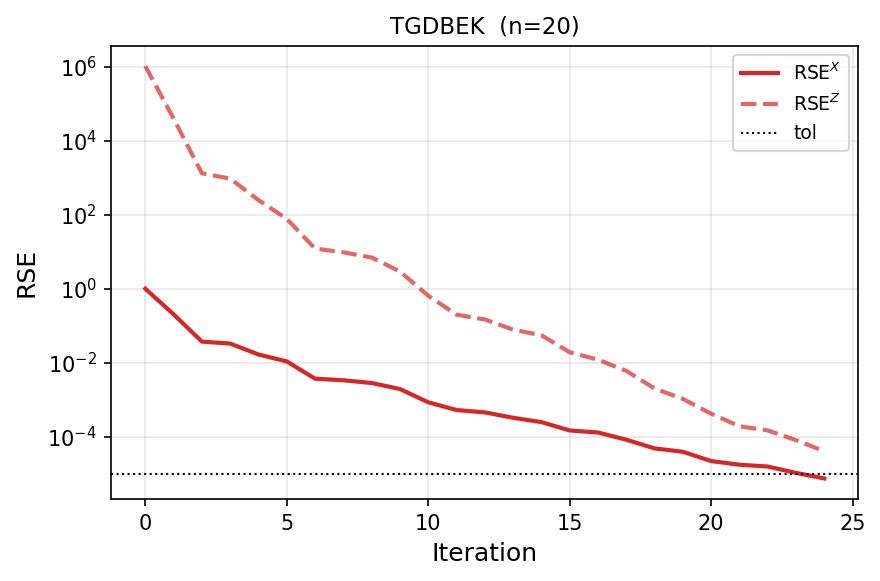}
  \caption{$\mathrm{RSE}$ for $\fm Z$ vs.\ iteration, $n=20$.}
  \label{fig:ex1-z-n20}
\end{subfigure}\hfill
\begin{subfigure}[t]{0.48\linewidth}
  \includegraphics[width=\linewidth]{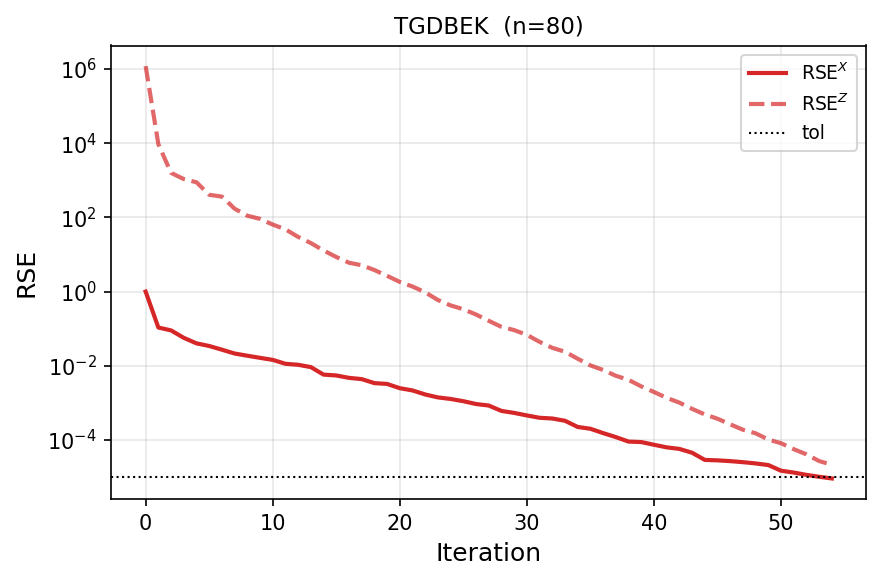}
  \caption{$\mathrm{RSE}$ for $\fm Z$ vs.\ iteration, $n=80$.}
  \label{fig:ex1-z-n80}
\end{subfigure}
\caption{Z-iterate relative squared error ($\mathrm{RSE}^Z$) versus iteration for $n=20$ (left) and $n=80$ (right). TGDBEK drives $\mathrm{RSE}^Z$ to zero faster than all competitors; the Z-iterate convergence mirrors the X-iterate pattern. Dotted line: tolerance $\mathrm{tol}=10^{-5}$.}\label{fig:ex1-z-rse}
\end{figure}

\begin{figure}[!ht]
\centering
\begin{subfigure}[t]{0.48\linewidth}
  \includegraphics[width=\linewidth]{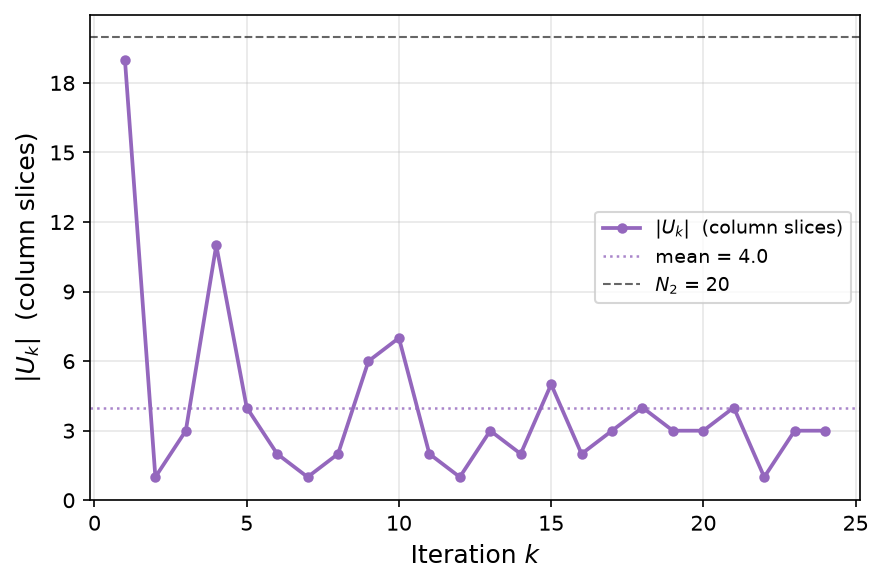}
  \caption{$|U_k|$, $n=20$.}\label{fig:ex1-blocks-n20-U}
\end{subfigure}\hfill
\begin{subfigure}[t]{0.48\linewidth}
  \includegraphics[width=\linewidth]{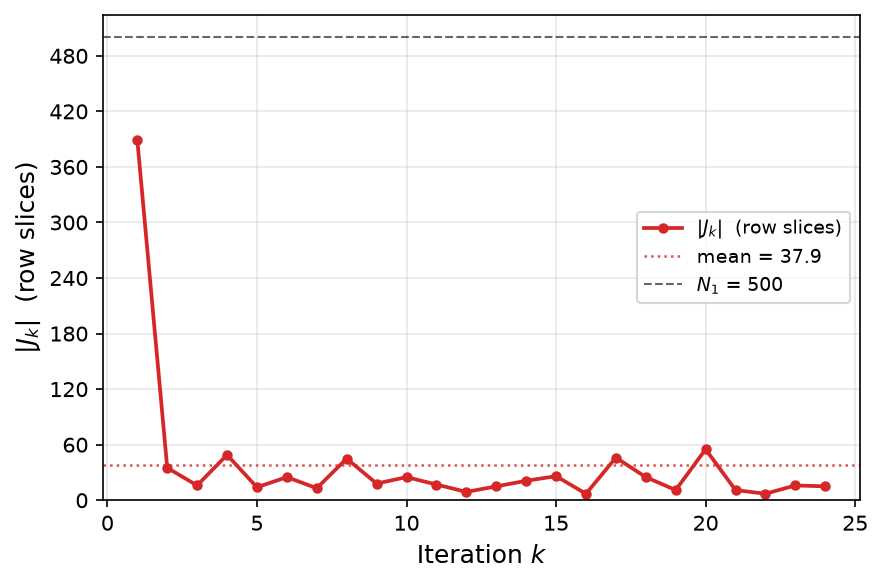}
  \caption{$|J_k|$, $n=20$.}\label{fig:ex1-blocks-n20-J}
\end{subfigure}
\\[0.6em]
\begin{subfigure}[t]{0.48\linewidth}
  \includegraphics[width=\linewidth]{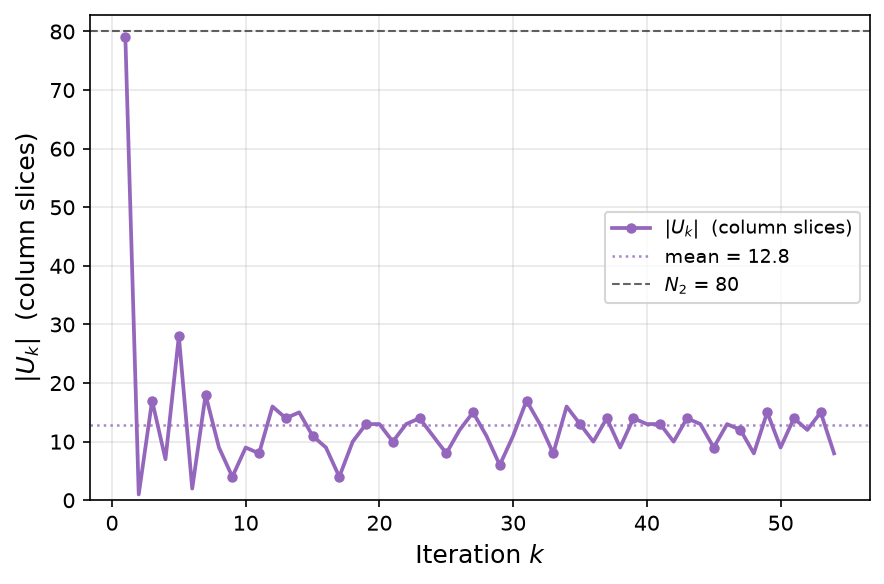}
  \caption{$|U_k|$, $n=80$.}\label{fig:ex1-blocks-n80-U}
\end{subfigure}\hfill
\begin{subfigure}[t]{0.48\linewidth}
  \includegraphics[width=\linewidth]{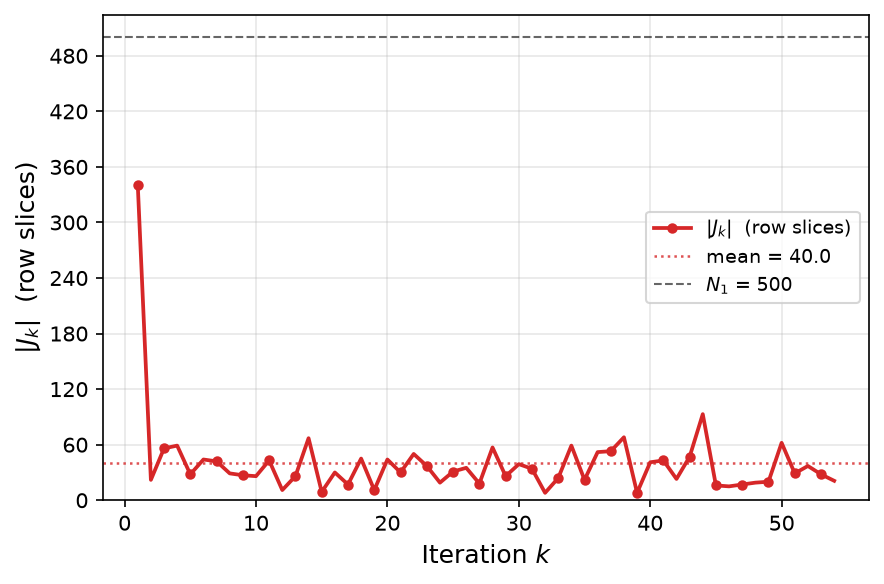}
  \caption{$|J_k|$, $n=80$.}\label{fig:ex1-blocks-n80-J}
\end{subfigure}
\caption{Cardinalities of each of the greedy blocks $U_k$ and $J_k$ versus iteration for dense systems $\fm A\in\mathbb{R}^{500\times n\times 10}$ ($\eta=0.6$, $a=10^{-3}$). Column set $|U_k|$ (left) and row set $|J_k|$ (right), for $n=20$ (top) and $n=80$ (bottom). Dotted line: mean over the run; dashed line: maximal admissible size ($N_2=n$ for $U_k$, $N_1=500$ for $J_k$).}\label{fig:ex1-blocks}
\end{figure}

\begin{example}[\textbf{Sparse}]\label{ex:sparse}
In this example, we use the TREK, TREBK, TREGBK, and TGDBEK methods to solve the tensor linear system~\eqref{eq:eq1}, where the coefficient tensor $\fm A$ is a sparse tensor. The sparse tensor here  $\fm A$ is constructed by reading five matrices from the SuiteSparse Matrix Collection~\cite{suiteSparse} and reshaping them into third-order tensors: \texttt{nos5} ($468\times39\times12$), \texttt{ash85} ($85\times17\times5$), \texttt{Cities} ($55\times23\times2$), \texttt{WorldCities} ($315\times20\times5$), and \texttt{gre\_216a} ($216\times24\times9$). The exact solution tensor $\fm X_\ast$ is randomly generated with standard Gaussian entries, the data tensor $\fm B$ is obtained by $\fm B = \fm A * \fm X_\ast + \fm\epsilon$, where $\fm\epsilon$ is generated as described above with noise level reported in Table~\ref{tab:ex2}, and the termination criterion is the same as Example~\ref{ex:dense}. The greedy threshold is $\eta=0.5$ for \texttt{nos5}, \texttt{ash85}, and \texttt{Cities}, and $\eta=0.7$ for \texttt{WorldCities} and \texttt{gre\_216a}. The number of blocks for TREBK  and TREGBK is determined by the dimension of the tensor $\fm A$.

For the different sparse tensor systems, IT, CPU, and RSE for the four methods when the desired accuracy is reached are given in Table~\ref{tab:ex2}. An iterative convergence diagram is given in Fig.~\ref{fig:ex2-nos5} for the \texttt{nos5} system. In the Fig.~\ref{fig:ex2-nos5}, the TREK method exhausted the total number of iterations, without achieving the required tolerance.

\begin{table}[!ht]
\caption{IT, CPU, and RSE for sparse tensor systems from SuiteSparse. Each block header gives tensor dimensions $N_1\times N_2\times N_3$ and density $\rho$ of the underlying sparse matrix. NaN: method failed to converge in the given iteration budget. Noise: \texttt{nos5} uses  $a=10^{-2}$; \texttt{ash85}, \texttt{Cities}, \texttt{gre\_216a} $a=10^{-3}$; \texttt{WorldCities} $a=10^{-1}$. Max iterations: $300$ (\texttt{nos5}, \texttt{ash85}); $6000$ (others). Bold: best per matrix.}\label{tab:ex2}
\centering\small\setlength{\tabcolsep}{4pt}
\begin{tabular*}{\textwidth}{@{\extracolsep{\fill}}llrrr@{}}
\toprule
Matrix & Method & IT & CPU (s) & RSE \\
\midrule
\multicolumn{5}{@{}l}{\texttt{nos5} \quad $468\times39\times12$,\; $\rho=2.36\%$} \\[1pt]
  & TREK            &  300 & 1.249 & $9.288\times10^{-2}$ \\
  & TREBK           &  138 & 4.361 & $9.888\times10^{-7}$ \\
  & TREGBK          &  106 & 3.401 & $9.559\times10^{-7}$ \\
  & \textbf{TGDBEK} &  \textbf{44} & \textbf{2.377} & $\mathbf{9.659\times10^{-7}}$ \\
\midrule
\multicolumn{5}{@{}l}{\texttt{ash85} \quad $85\times17\times5$,\; $\rho=7.24\%$} \\[1pt]
  & TREK            &  300 & 1.123 & $1.386\times10^{-1}$ \\
  & TREBK           &  171 & 2.315 & $9.463\times10^{-7}$ \\
  & TREGBK          &  300 & 4.117 & $1.028\times10^{-6}$ \\
  & \textbf{TGDBEK} & \textbf{128} & \textbf{2.395} & $\mathbf{9.716\times10^{-7}}$ \\
\midrule
\multicolumn{5}{@{}l}{\texttt{Cities} \quad $55\times23\times2$,\; $\rho=53.04\%$} \\[1pt]
  & TREK            & 6000 & 20.947 & $3.807\times10^{-3}$ \\
  & TREBK           & 1916 & 13.347 & $9.815\times10^{-7}$ \\
  & TREGBK          & 5740 & 45.566 & $9.881\times10^{-7}$ \\
  & \textbf{TGDBEK} & \textbf{1107} & \textbf{15.627} & $\mathbf{9.783\times10^{-7}}$ \\
\midrule
\multicolumn{5}{@{}l}{\texttt{WorldCities} \quad $315\times20\times5$,\; $\rho=23.87\%$} \\[1pt]
  & TREK            & 6000 & 21.676 & $\mathrm{NaN}$ \\
  & TREBK           & 1136 & 15.599 & $9.888\times10^{-7}$ \\
  & TREGBK          &  623 &  9.274 & $9.875\times10^{-7}$ \\
  & \textbf{TGDBEK} & \textbf{262} & \textbf{5.531} & $\mathbf{9.970\times10^{-7}}$ \\
\midrule
\multicolumn{5}{@{}l}{\texttt{gre\_216a} \quad $216\times24\times9$,\; $\rho=1.88\%$} \\[1pt]
  & TREK            &  879 & 3.148 & $9.893\times10^{-7}$ \\
  & TREBK           &  100 & 2.316 & $8.839\times10^{-7}$ \\
  & TREGBK          &   72 & 1.846 & $9.749\times10^{-7}$ \\
  & \textbf{TGDBEK} &  \textbf{61} & \textbf{2.176} & $\mathbf{9.192\times10^{-7}}$ \\
\bottomrule
\end{tabular*}
\end{table}
\begin{figure}[ht]
\centering
\includegraphics[width=\linewidth]{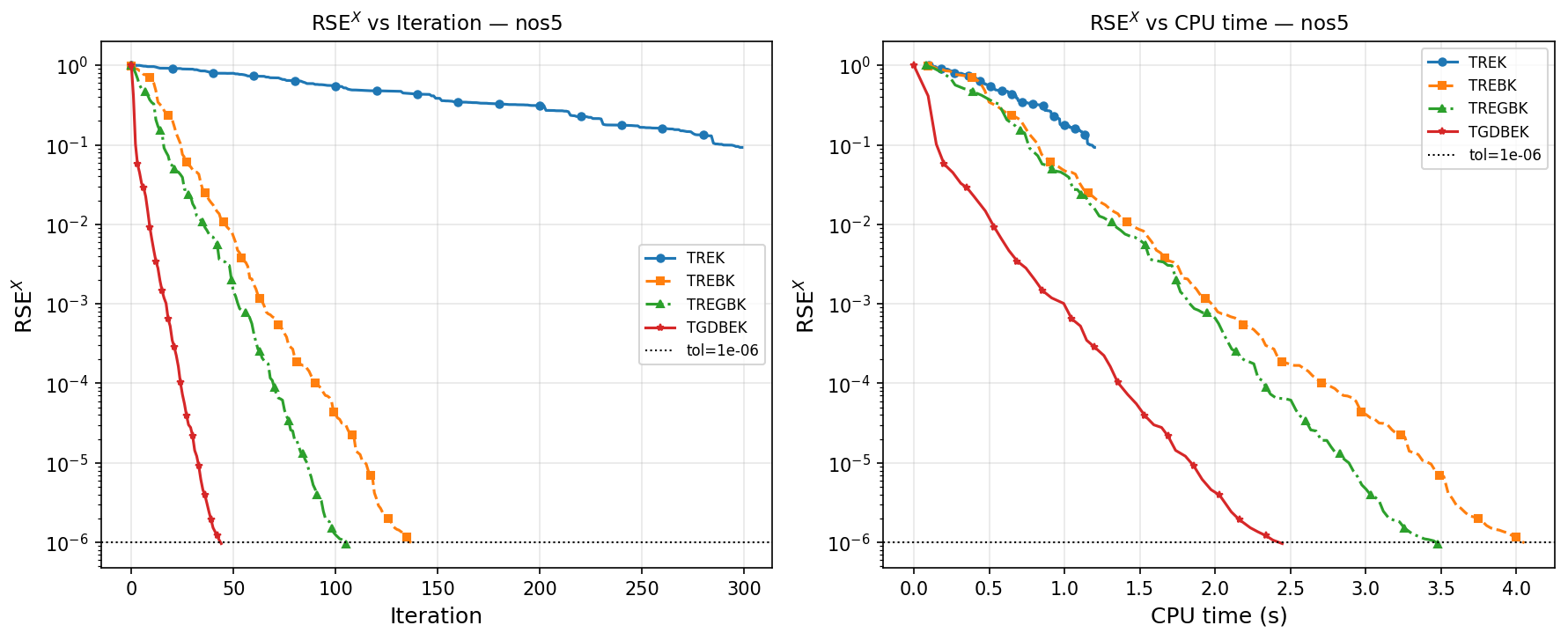}
\caption{$\mathrm{RSE}$ versus iteration (left) and $\mathrm{RSE}$ versus wall-clock CPU time (right) for the \texttt{nos5} sparse tensor system ($468\times39\times12$, $a=10^{-2}$, $T=300$). TGDBEK converges in 44 iterations and reaches the tolerance first in CPU time; TREK exhausts the budget. Dotted line: tolerance $\mathrm{tol}=10^{-6}$.}\label{fig:ex2-nos5}
\end{figure}
\end{example}

\FloatBarrier
\begin{example}[\textbf{Color image deblurring}]\label{ex:color-image}
In this example, we use the TREK, TREBK, TREGBK, and TGDBEK methods to solve the tensor linear system~\eqref{eq:eq1} arising from color image restoration. The test image is a public-domain flower photograph, resized and stored as a tensor $\fm X_\ast \in \mathbb{R}^{200\times200\times3}$. The deblurring operator $\fm A \in \mathbb{R}^{200\times200\times3}$ is constructed by setting the first frontal slice $\fm A_{:,:,1}$ to the normalized Gaussian Toeplitz blur matrix with $\sigma=4$ and bandwidth $\mathrm{band}=32$, and all remaining frontal slices to zero.  In fact, we let
\begin{align*}
  z_j = \exp\!\left(-\frac{j^2}{2\sigma^2}\right), \quad j = 0, \ldots, \mathrm{band}-1,
  \qquad
  \fm A_{:,:,1} = \frac{1}{\sigma\pi}\,D^{-1}T(z), \quad D = \mathrm{diag}(T(z)\,\mathbf{1}),
\end{align*}
where $T(z)$ is the symmetric Toeplitz matrix generated by $z$, $\mathbf{1}$ is the all-ones vector, and $D^{-1}T(z)$ denotes row normalization, with parameters $\sigma=4$ and $\mathrm{band}=32$.  The data tensor is obtained by $\fm B = \fm A * \fm X_\ast + \fm\epsilon$, where $\fm\epsilon$ is generated as described above with noise level $a=10^{-2}$, and the termination criterion is the same as Example~\ref{ex:dense}. The maximum number of iterations is $1000$.

The final experimental results are given in Table~\ref{tab:ex3}. The original image, the blurred and noisy image, and the images recovered by each method are shown in Fig.~\ref{fig:ex3-summary}. The $\mathrm{RSE}$ and $\mathrm{RSE}^Z$ versus iteration are shown in Fig.~\ref{fig:ex3-(g)}.
\end{example}

\begin{figure}[!ht]
\centering
\begin{subfigure}[t]{0.30\linewidth}
  \includegraphics[width=\linewidth]{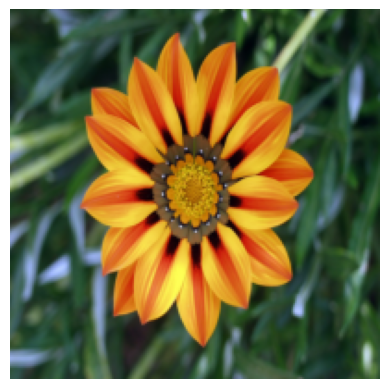}
  \caption{True image}\label{fig:ex3-(a)}
\end{subfigure}\hfill
\begin{subfigure}[t]{0.30\linewidth}
  \includegraphics[width=\linewidth]{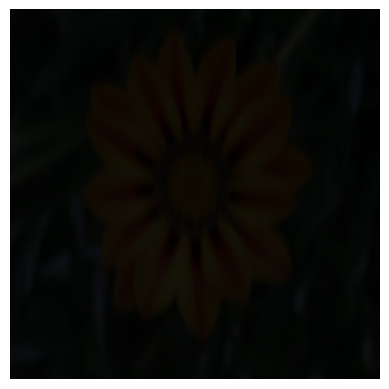}
  \caption{Blurred + noisy}\label{fig:ex3-(b)}
\end{subfigure}\hfill
\begin{subfigure}[t]{0.30\linewidth}
  \includegraphics[width=\linewidth]{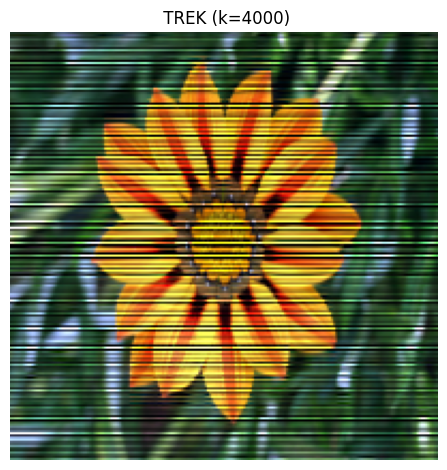}
  \caption{TREK}\label{fig:ex3-(c)}
\end{subfigure}
\\[0.6em]
\begin{subfigure}[t]{0.30\linewidth}
  \includegraphics[width=\linewidth]{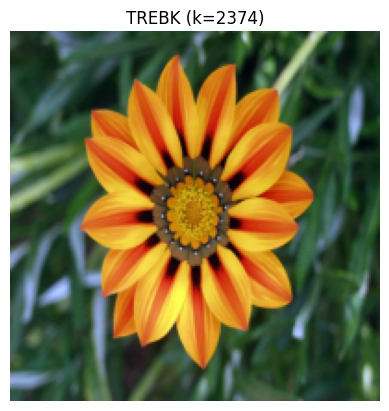}
  \caption{TREBK}\label{fig:ex3-(d)}
\end{subfigure}\hfill
\begin{subfigure}[t]{0.30\linewidth}
  \includegraphics[width=\linewidth]{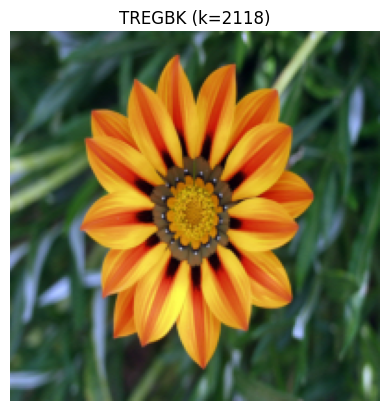}
  \caption{TREGBK}\label{fig:ex3-(e)}
\end{subfigure}\hfill
\begin{subfigure}[t]{0.30\linewidth}
  \includegraphics[width=\linewidth]{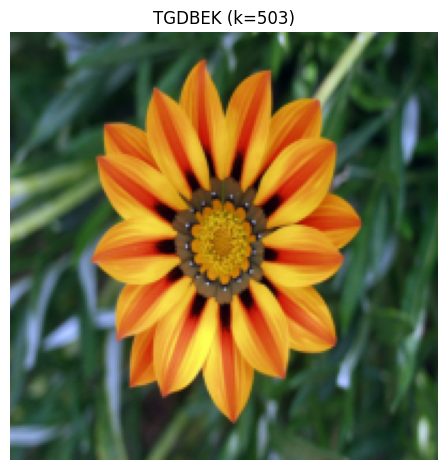}
  \caption{TGDBEK}\label{fig:ex3-(f)}
\end{subfigure}
\caption{Color ``flower'' image restoration: (a) true image, (b) blurred and noisy, (c)--(f) recovered images by TREK, TREBK, TREGBK, and TGDBEK.}\label{fig:ex3-summary}
\end{figure}

\begin{figure}[!ht]
\centering
\includegraphics[width=0.75\linewidth]{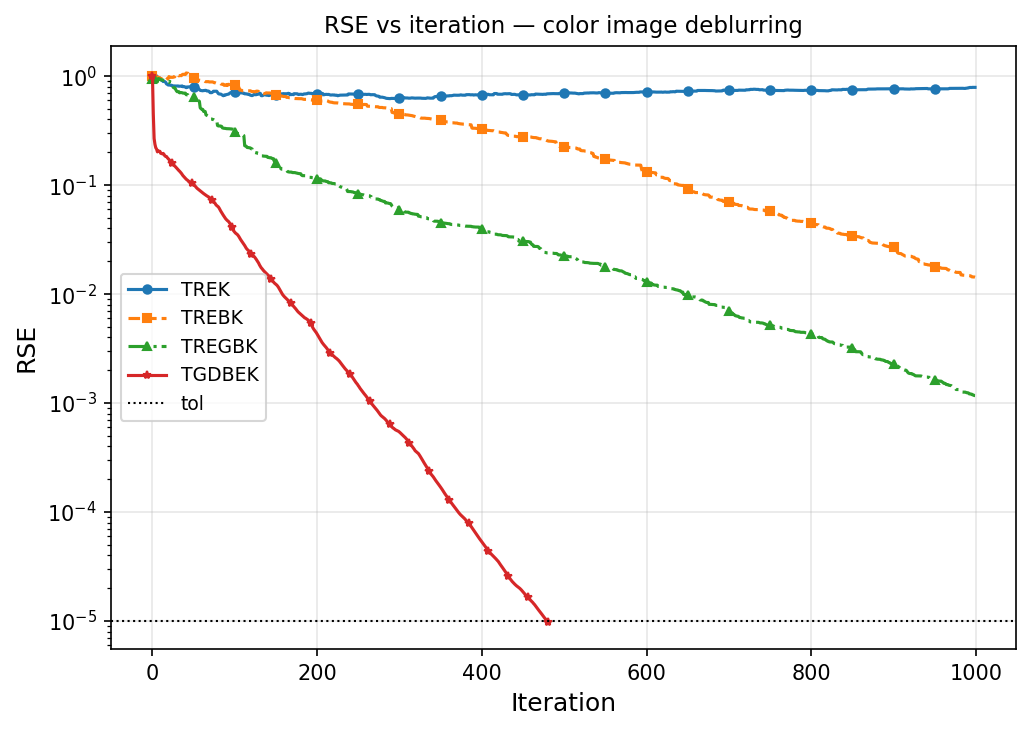}
\caption{$\mathrm{RSE}$ versus iteration for color image deblurring ($200\times200\times3$, $\sigma=4$, $a=10^{-2}$, max $1000$ iterations). TGDBEK is the only method to reach tolerance $10^{-5}$, converging at iteration $503$.}\label{fig:ex3-(g)}
\end{figure}

\begin{table}[!ht]
\centering\small
\caption{CPU, IT, RSEs, for color image deblurring ($200\times200\times3$, $\sigma=4$, $a=10^{-2}$, max $1000$ iterations). The last column gives the RSE for the $\fm Z$ iterate.}\label{tab:ex3}
\begin{tabular}{lrrrr}
\toprule
Method & CPU (s) & IT & RSE & RSE \\
\midrule
TREK   &  2.088 & 1000 & $7.92\times10^{-1}$ & --- \\
TREBK  &  8.822 & 1000 & $2.18\times10^{-2}$ & --- \\
TREGBK &  8.216 & 1000 & $1.27\times10^{-3}$ & --- \\
TGDBEK &  6.006 &  503 & $1.00\times10^{-5}$ & $8.745\times10^{-10}$ \\
\bottomrule
\end{tabular}
\end{table}

As can be seen from Table~\ref{tab:ex3}, TREK and TREBK exhaust the $1000$-iteration budget without reaching the prescribed tolerance, stopping at $\mathrm{RSE}^X = 7.92\times10^{-1}$ and $2.18\times10^{-2}$, respectively. TREGBK also fails to converge, reaching only $\mathrm{RSE}^X = 1.27\times10^{-3}$. TGDBEK is the only method to reach $\mathrm{RSE}^X < 10^{-5}$, converging at iteration $503$ in $6.01$, a considerable reduction in wall-clock time compared to TREGBK ($8.22\,\mathrm{s}$). Moreover, from Fig.~\ref{fig:ex3-summary}, the image recovered by TGDBEK is slightly sharper than those of all competing methods. Fig.~\ref{fig:ex3-(g)} further confirms TGDBEK's superior convergence rate across all methods.

\FloatBarrier

\begin{example}[\textbf{Gray image deblurring}]\label{ex:gray}
In this example, we apply the TREK, TREBK, TREGBK, and TGDBEK methods to solve the tensor linear system~\eqref{eq:eq1} arising from gray MRI-like image deblurring. The test image is a Shepp--Logan phantom~\cite{shepp1974fourier} resized to $128\times128$ pixels. A volumetric dataset $\fm X_\ast \in \mathbb{R}^{128\times128\times27}$ is constructed by generating $27$ frontal slices, each obtained by rotating the base phantom by an angle uniformly spaced in $[-8^\circ, 8^\circ]$ and applying a mild sinusoidal intensity scaling across slices. The deblurring operator $\fm A \in \mathbb{R}^{128\times128\times27}$ is constructed as follows. We let
\begin{align*}
  z_j = \exp\!\left(-\frac{j^2}{2\sigma^2}\right), \quad j = 0, \ldots, \mathrm{band}-1,
  \qquad
  \fm A_{:,:,1} = \frac{1}{\sigma\pi}\,D^{-1}T(z), \quad D = \mathrm{diag}(T(z)\,\mathbf{1}),
\end{align*}
where $T(z)$ is the symmetric Toeplitz matrix generated by $z$, $\mathbf{1}$ is the all-ones vector, and $D^{-1}T(z)$ denotes row normalization, with parameters $\sigma=4$ and $\mathrm{band}=64$. All remaining frontal slices are set to zero. The data tensor is obtained by $\fm B = \fm A * \fm X_\ast + \fm\epsilon$, where $\fm\epsilon$ is generated as described above with noise level $a=10^{-3}$, and the termination criterion is the same as Example~\ref{ex:dense}. The maximum number of iterations is $1000$.

The final experimental results are given in Table~\ref{tab:ex5}. The blurred image and the images recovered by each method (central frontal slice) are shown in Fig.~\ref{fig:gray-summary}.
\end{example}
\begin{table}[!ht]
\caption{IT, CPU, PSNR, and SSIM for gray MRI-like image deblurring ($128\times128\times27$, $\sigma=4$, $a=10^{-3}$, max $1000$ iterations). PSNR and SSIM are computed on the central frontal slice.}\label{tab:ex5}
\centering\small\setlength{\tabcolsep}{5pt}
\begin{tabular}{lrrrr}
\toprule
Method & IT & CPU (s) & PSNR (dB) & SSIM \\
\midrule
TREK            & 1000 & 1.523 & 18.14 & 0.602 \\
TREBK           & 1000 & 3.041 & 57.74 & 0.998 \\
TREGBK          & 1000 & 2.654 & 57.99 & 0.998 \\
\textbf{TGDBEK} & \textbf{1000} & \textbf{2.013} & \textbf{57.99} & \textbf{0.998} \\
\bottomrule
\end{tabular}
\end{table}

\begin{figure}[!ht]
\centering
\begin{subfigure}[t]{0.30\linewidth}
  \includegraphics[width=\linewidth]{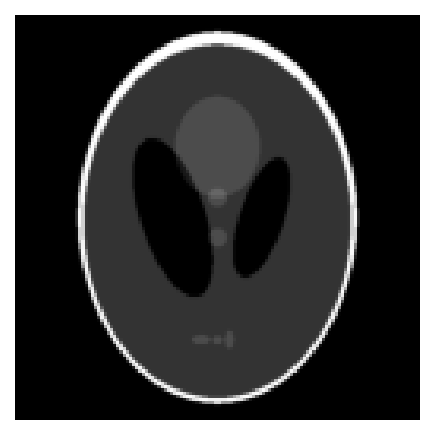}
  \caption{True image}\label{fig:gray-original}
\end{subfigure}\hfill
\begin{subfigure}[t]{0.30\linewidth}
  \includegraphics[width=\linewidth]{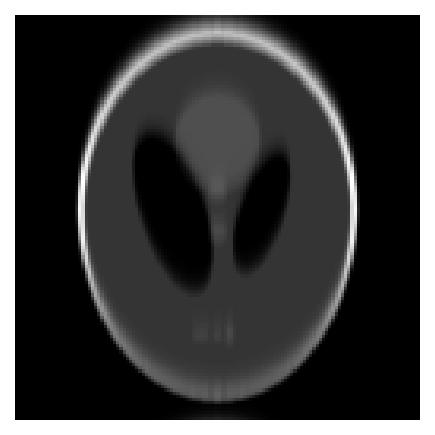}
  \caption{Blurred + noisy}\label{fig:gray-blur}
\end{subfigure}\hfill
\begin{subfigure}[t]{0.30\linewidth}
  \includegraphics[width=\linewidth]{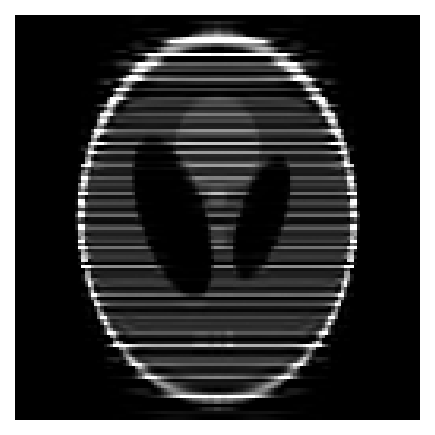}
  \caption{TREK}
\end{subfigure}
\\[0.6em]
\begin{subfigure}[t]{0.30\linewidth}
  \includegraphics[width=\linewidth]{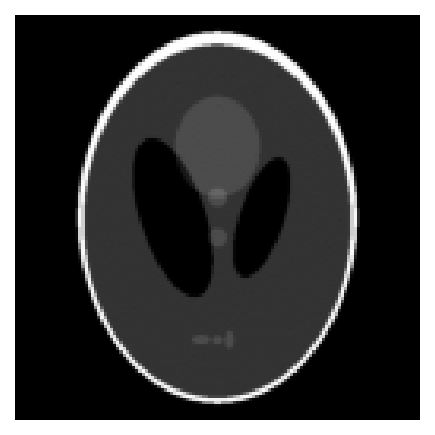}
  \caption{TREBK}
\end{subfigure}\hfill
\begin{subfigure}[t]{0.30\linewidth}
  \includegraphics[width=\linewidth]{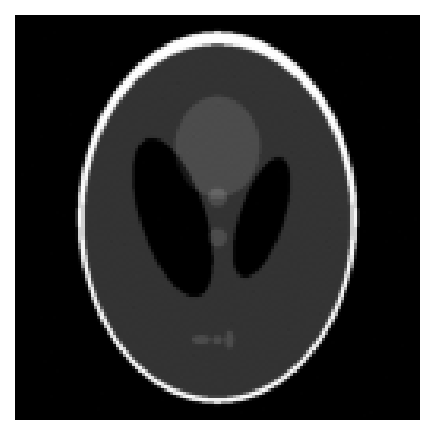}
  \caption{TREGBK }
\end{subfigure}\hfill
\begin{subfigure}[t]{0.30\linewidth}
  \includegraphics[width=\linewidth]{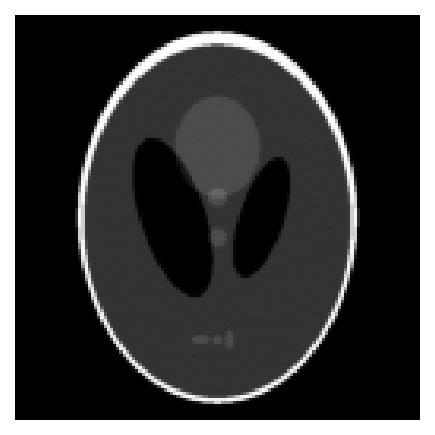}
\caption{TGDBEK }
\end{subfigure}
\caption{Gray MRI-like image deblurring ($128\times128\times27$, central frontal slice): (a) true image, (b) blurred and noisy, (c)--(f) recovered images by TREK, TREBK, TREGBK, and TGDBEK. TREK diverges; TREBK, TREGBK, and TGDBEK achieve near-identical high-fidelity recovery.}\label{fig:gray-summary}
\end{figure}

As can be seen from Table~\ref{tab:ex5} and Fig.~\ref{fig:gray-summary},  all the algorithms exhaust the given total number of iterations, with TREK  being the first to stop  while it didn't  achieve the prescribe tolerance. TREK recoves poorly the original image with $\mathrm{PSNR} = 18.14\,\mathrm{dB}$, $\mathrm{SSIM} = 0.602$. TREBK, TREGBK, and TGDBEK all recover the phantom faithfully, achieving $\mathrm{PSNR} \approx 58\,\mathrm{dB}$ and $\mathrm{SSIM} \approx 0.998$.  Notably, TGDBEK  matches the reconstruction quality of TREGBK  while achieving faster the tolerance in fewer seconds than TREGBK. That demonstrates how the greedy double-block strategy yields computational savings even when all methods exhaust the iteration budget.

\FloatBarrier
\begin{example}[\textbf{Effect of the threshold $\eta$}]\label{ex:eta}
In this example, we examine how the greedy threshold $\eta$ affects the convergence speed and computational cost of TGDBEK. Recall that $\eta \in (0,1]$ controls the size of the active-sets: a row block $i$ (resp.\ column block $j$) is included in the active set if and only if its squared residual contribution is at least $\eta$ times the maximum over all blocks. A small $\eta$ admits many blocks per step (akin to a full block update), whereas $\eta = 1$ retains only the single most informative block at each iteration (akin to TREK).

Here, the test system has coefficient tensor $\fm A \in \mathbb{R}^{500\times20\times10}$ with a randomly generated exact solution $\fm X_\ast \in \mathbb{R}^{20\times10\times10}$ and noise level $a = 10^{-2}$. The experiment is run for each value $\eta \in \{0.1, 0.2, \ldots, 1.0\}$, with a maximum of $2000$ iterations and tolerance $10^{-6}$, and results are averaged over $20$ independent trials with standard Gaussian entries.

The average iteration count IT and CPU time as functions of $\eta$ are reported in Table~\ref{tab:eta} and Fig.~\ref{fig:eta}.

\begin{table}[!ht]
\caption{Average IT and CPU of TGDBEK vs.\ greedy threshold $\eta$ ($\fm A\in\mathbb{R}^{500\times20\times10}$, tol $=10^{-6}$, 20 runs) at two noise levels.}\label{tab:eta}
\centering\small\setlength{\tabcolsep}{8pt}
\begin{tabular}{crr crr}
\toprule
& \multicolumn{2}{c}{$a=10^{-2}$} & & \multicolumn{2}{c}{$a=10^{-1}$} \\
\cmidrule{2-3}\cmidrule{5-6}
$\eta$ & IT & CPU (s) & & IT & CPU (s) \\
\midrule
$0.1$ &   1 & 0.031 & &   1 & 0.031 \\
$0.2$ &   1 & 0.027 & &   1 & 0.027 \\
$0.3$ &   1 & 0.027 & &   1 & 0.027 \\
$0.4$ &   1 & 0.029 & &   1 & 0.029 \\
$0.5$ &   6 & 0.264 & &   6 & 0.273 \\
$0.6$ &  18 & 0.674 & &  18 & 0.650 \\
$0.7$ &  37 & 1.273 & &  35 & 1.184 \\
$0.8$ &  57 & 1.796 & &  58 & 1.716 \\
$0.9$ &  88 & 2.441 & &  88 & 2.343 \\
$1.0$ & 362 & 7.971 & & 362 & 7.898 \\
\bottomrule
\end{tabular}
\end{table}

We can see that from Table~\ref{tab:eta} and Fig.~\ref{fig:eta}, the behavior of TGDBEK is remarkably consistent across both noise levels. For $\eta \leq 0.4$, TGDBEK converges in a single iteration at both noise levels. At $\eta = 0.5$ the IT rises to $6$, marking the onset of the transition regime. As $\eta$ increases beyond $0.5$, the active set shrinks and both IT and CPU time grow substantially. At $\eta = 1.0$, the algorithm degenerates to a TREK-like single-block update requiring $362$ iterations — approximately $20\times$ more than at $\eta = 0.6$. The recommended range is $\eta \in [0.1, 0.6]$, where convergence speed and per-iteration cost are jointly minimized, and this recommendation holds regardless of noise level.
\end{example}

\begin{figure}[!htbp]
\centering
\begin{subfigure}[t]{0.90\linewidth}
  \centering
  \includegraphics[width=\linewidth]{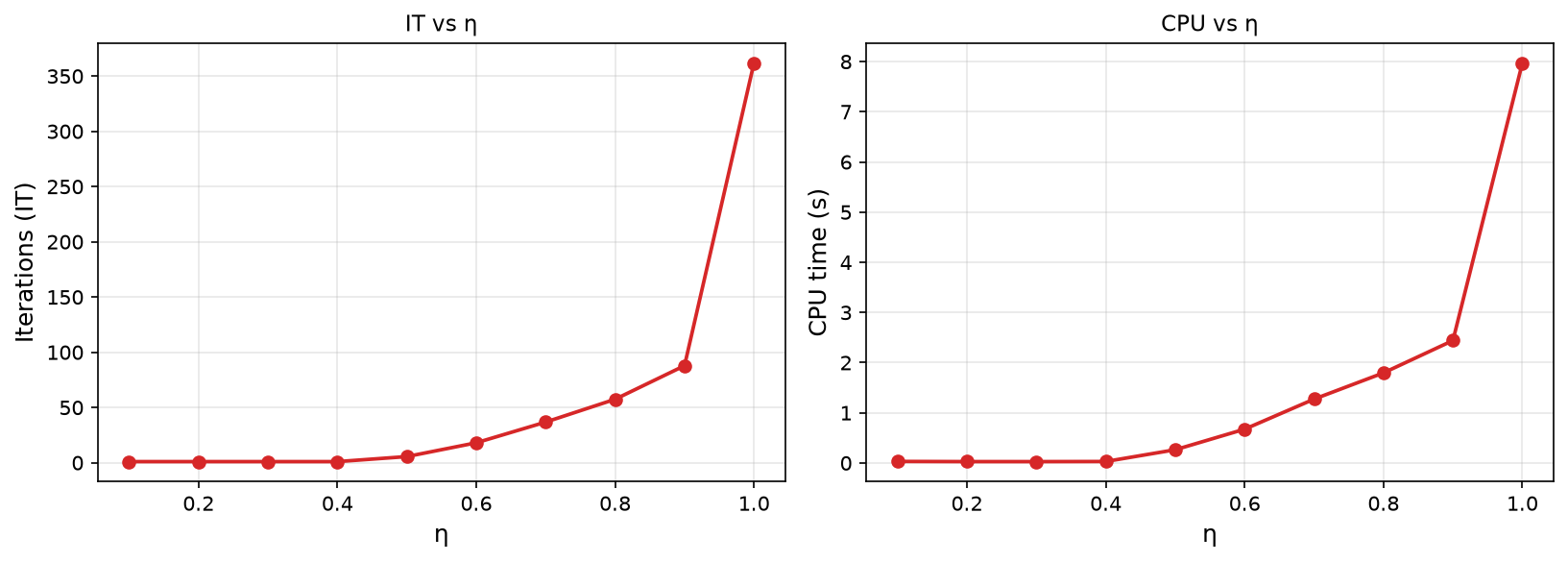}
  \caption{Average IT (left) and CPU time (right) vs.\ $\eta$.}\label{fig:eta-conv}
\end{subfigure}
\\[0.6em]
\begin{subfigure}[t]{0.48\linewidth}
  \includegraphics[width=\linewidth]{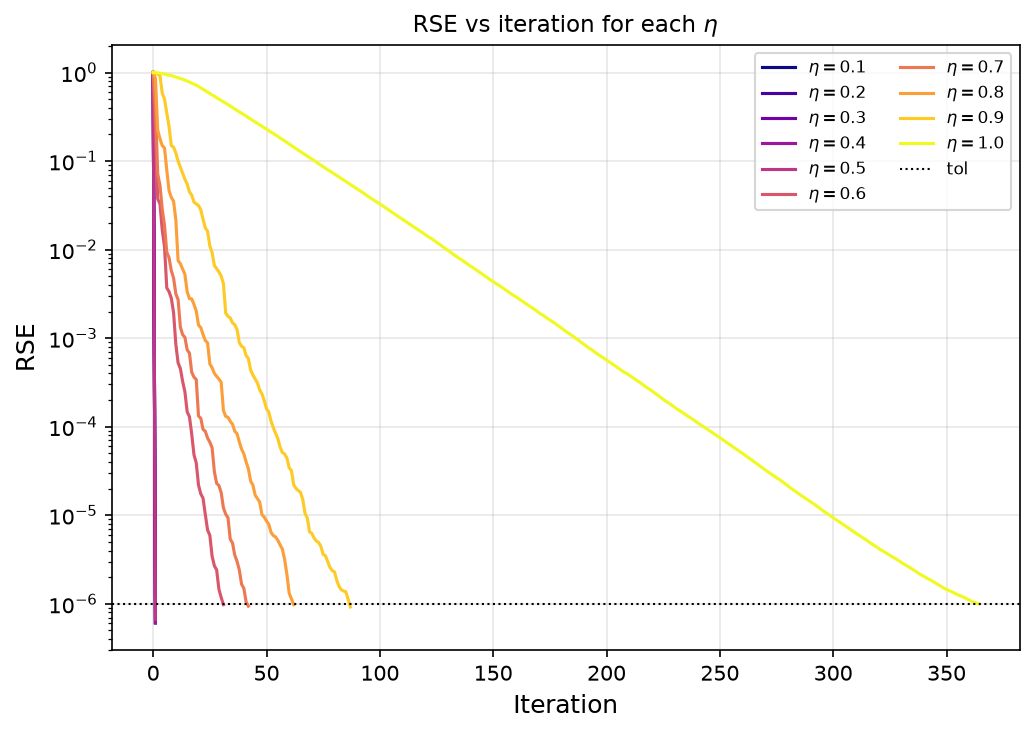}
  \caption{RSE vs.\ iteration, $a=10^{-2}$.}\label{fig:eta-rse-low}
\end{subfigure}\hfill
\begin{subfigure}[t]{0.48\linewidth}
  \includegraphics[width=\linewidth]{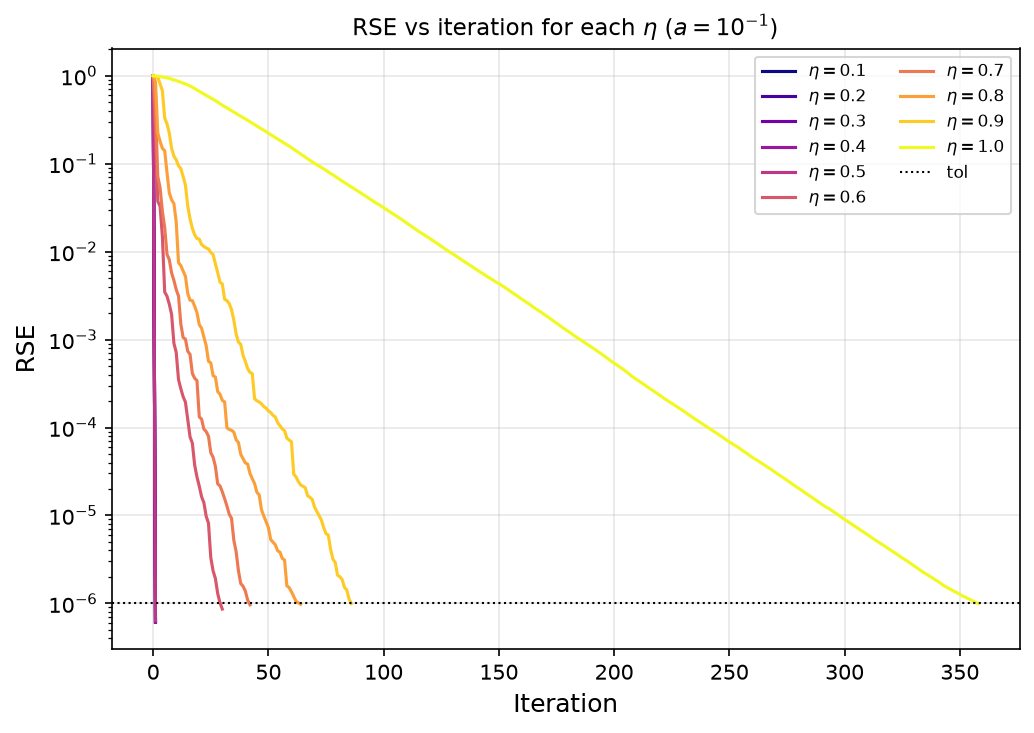}
  \caption{RSE vs.\ iteration, $a=10^{-1}$.}\label{fig:eta-rse-high}
\end{subfigure}
\caption{Effect of the greedy threshold $\eta$ on TGDBEK. Top: average IT and CPU vs.\ $\eta$ ($a=10^{-2}$, 20 runs). Bottom: RSE vs.\ iteration for each $\eta\in\{0.1,\ldots,1.0\}$ at two noise levels ($\fm A\in\mathbb{R}^{500\times20\times10}$).}\label{fig:eta}
\end{figure}

As can be seen from Table~\ref{tab:eta} and Fig.~\ref{fig:eta}, the sharp transition between $\eta=0.4$ (IT$=1$) and $\eta=0.5$ (IT$=6$) reflects a qualitative change in block size. The practical recommendation is to keep $\eta\in[0.1,0.5]$ when iteration count is the priority, or $\eta\approx0.5$--$0.6$ when per-step cost matters.

\section{Conclusion}\label{sec:concl}


In this paper, we proposed the Tensor Greedy Double Block Extended Kaczmarz (TGDBEK) method for solving large, inconsistent tensor linear systems under the t-product algebraic framework. By dynamically selecting active sets of horizontal and lateral slices with maximal residual energy at each iteration, TGDBEK eliminates the need for predefined, rigid block partitions used in existing methods such as TREBK and TREGBK. We established a theoretical convergence analysis proving that TGDBEK converges linearly in the Frobenius norm to the unique minimum-norm least-squares solution $\mathcal{A}^\dagger * \mathcal{B}$. 

Extensive numerical experiments across synthetic dense overdetermined tensor systems, real-world sparse tensors from the SuiteSparse collection, multichannel color image deblurring, and 3D volumetric MRI phantom reconstruction demonstrate the superiority of TGDBEK. Across all benchmarks, TGDBEK achieved up to a $5\times$ reduction in iteration count and up to a $2.5\times$ reduction in wall-clock CPU time compared to existing state-of-the-art tensor Kaczmarz solvers. 

Several promising research directions arise from this work:
\begin{enumerate}
    \item \textbf{Projection-Free (Pseudoinverse-Free) Variants by Averaging\cite{necoara2019faste,tondji2023faster}:} In each step of TGDBEK, computing the Moore--Penrose pseudoinverse $(\mathcal{A}_{J_k,:,:})^\dagger$ of active sub-tensors can become computationally demanding for very large block sizes. As a natural next step, we will develop \emph{projection-free greedy extended block Kaczmarz methods} leveraging an \emph{averaging strategy} (in the spirit of tensor randomized extended average block Kaczmarz). By replacing tensor pseudoinverses with weighted parallel slice updates and averaging, the per-iteration computational complexity can be drastically curtailed while preserving fast convergence.
    \item \textbf{Adaptive and Dynamic Thresholding:} Investigating dynamic threshold sequences $\eta_k$ that adaptively shrink or expand active block sizes based on the decay of the residual spectrum.
    \item \textbf{High-Order Generalization:} Extending the TGDBEK and greedy averaging frameworks to arbitrary $p$-th order tensors ($p \ge 4$) under the generalized tuple t-product.
\end{enumerate}

\end{document}